\documentclass[a4paper,american,reqno]{amsart}

\usepackage{babel}
\usepackage[utf8]{inputenc}
\usepackage[T1]{fontenc}
\usepackage[binary-units=true]{siunitx}
\usepackage{csquotes}
\usepackage{todonotes}
\usepackage{booktabs}
\usepackage{forest}
\usepackage{tikz}
\usepackage{longtable}
\usepackage[export]{adjustbox} 
\usepackage{array}    
\usepackage[style = authoryear-comp,
            maxbibnames = 100,
            maxcitenames = 2,
            giveninits = true,
            uniquename = init,
            isbn = false,
            backend = bibtex]{biblatex}
\usepackage[colorlinks,
            citecolor=blue,
            urlcolor=blue,
            linkcolor=blue]{hyperref} 
\usepackage{subcaption}

\makeatletter
\patchcmd{\@settitle}{\uppercasenonmath\@title}{\scshape\large}{}{}
\patchcmd{\@setauthors}{\MakeUppercase}{\scshape\normalsize}{}{}
\makeatother

\vfuzz \hfuzz
\makeatletter
\@namedef{subjclassname@2020}{%
  \textup{2020} Mathematics Subject Classification}
\makeatother

\newtheorem{defi}{Definition}

\newtheorem{prop}{Proposition}

\newtheorem{thm}{Theorem}
\newtheorem{lem}[thm]{Lemma}

\newcommand{\st}{\text{s.t.}}
\newcommand{\define}{\mathrel{{\mathop:}{=}}}

\makeatletter
\newenvironment{varsubequations}[1]
{%
  \addtocounter{equation}{-1}%
  \begin{subequations}
    \renewcommand{\theparentequation}{#1}%
    \def\@currentlabel{#1}%

  }
  {%
  \end{subequations}\ignorespacesafterend
}
\makeatother

\newcommand\MyBox[2]{
  \fbox{\lower0.75cm
    \vbox to 1.7cm{\vfil
      \hbox to 1.7cm{\hfil\parbox{1.4cm}{#1\\#2}\hfil}
      \vfil}%
  }%
}

\newcommand{\defset}[3][\defsep]{\set{#2#1#3}}

\newcommand{\set}[1]{\{#1\}}
\newcommand{\Set}[1]{\left\{#1\right\}}

\newcommand{\field}{\mathbb}
\newcommand{\naturals}{\field{N}}

\newcommand{\reals}{\field{R}}

\newcommand{\N}{\naturals}

\newcommand{\R}{\reals}

\newcommand{\AC}{\text{AC}}
\newcommand{\TP}{\text{TP}}
\newcommand{\TN}{\text{TN}}
\newcommand{\FP}{\text{FP}}
\newcommand{\FN}{\text{FN}}
\newcommand{\PR}{\text{PR}}
\newcommand{\RE}{\text{RE}}

\newcommand{\OCTH}{\text{OCT-H}}

\newcommand{\LE}{\text{LE}}
\newcommand{\MCC}{\text{MCC}}
\newcommand{\SOCT}{\text{S$^2$OCT}}
\newcommand{\codename}[1]{\textsf{#1}}

\newcommand{\rev}[1]{#1}
\newcommand{\revTwo}[1]{\textcolor{black}{#1}}

\bibliography{semi-supervised-tree}

\begin{document}

\title[MILP for Semi-Supervised Optimal Classification Trees]%
{Mixed-Integer Linear Optimization\\
  for Semi-Supervised Optimal Classification Trees}

\author[J. P. Burgard, M. E. Pinheiro, M. Schmidt]%
{Jan Pablo Burgard, Maria Eduarda Pinheiro, Martin Schmidt}

\address[J. P. Burgard]{%
  Trier University,
  Department of Economic and Social Statistics,
  Universitätsring 15,
  54296 Trier,
  Germany}
\email{burgardj@uni-trier.de}

\address[M. E. Pinheiro, M. Schmidt]{%
  Trier University,
  Department of Mathematics,
  Universitätsring 15,
  54296 Trier,
  Germany}
\email{pinheiro@uni-trier.de}
\email{martin.schmidt@uni-trier.de}

\date{\today}

\begin{abstract}
  Decision trees are one of the most \rev{popular} methods for solving
classification problems, mainly because of their good
interpretability properties.
Moreover, due to advances in recent years in mixed-integer
optimization, several models have been proposed to formulate the
problem of computing optimal classification trees.
The goal is, given a set of labeled points, to split the feature space
with hyperplanes and assign a class to each \rev{part of the resulting
  partition}.
In certain scenarios, however, labels are \rev{only available} for a
subset of the given points.
Additionally, this subset may be non-representative, such as in the
case of self-selection in a survey.
Semi-supervised decision trees tackle the setting of labeled and
unlabeled data and often contribute to enhancing the reliability of
the results.
Furthermore, undisclosed sources may provide extra information about
the size of the classes.
We propose a mixed-integer linear optimization model for computing
semi-supervised optimal classification trees that cover
the setting of labeled and unlabeled data points as well as the
overall number of points in each class for a binary classification.
Our numerical results show that our approach leads to a better
accuracy and a better Matthews correlation coefficient for biased
samples compared to other optimal classification trees, even if only
few labeled points are available.


\end{abstract}

\keywords{Semi-supervised learning,
Optimal classification trees,
Mixed-integer linear optimization%
%
%
}
\subjclass[2020]{90C11,
90C90,
90-08,
68T99%
%
%
}

\maketitle

\section{Introduction}
\label{sec:introduction}

Decision trees are among the most popular approaches for supervised
classification \parencite{CART,Dtquilann}.
One of the main reasons for this is that they are easy to interpret
compared to other machine-learning models.
The core idea is to recursively partition the feature space, according
to branching rules, and assign a label to each \rev{part of the
  resulting partition}.

One way to partition the data is to use hyperplanes involving a
single feature, which leads to so-called univariate trees; see,
e.g., \textcite{univariateparallel, univariatehybrid}.
In a multivariate tree, these hyperplanes involve more than one
feature, and some approaches for this setting are given in
\textcite{multi1,bennett1996optimal,multi3}.
In many algorithms for univariate or multivariate trees, each separate
hyperplane is generated by minimizing a local impurity function, i.e.,
they do not build the tree by solving just a single optimization problem.

In recent years, due to the advancement of algorithms for
mixed-integer programming (MIP), many strategies for computing optimal
classification trees (OCT) by globally solving an optimization problem
using MIP techniques have been proposed.
Some techniques are discussed in the recent surveys by
\textcite{survey1,survey2}.
The first approaches were proposed by \textcite{Bertsimas2017}.
They present two mixed-integer linear programming (MILP) models based
on univariate and multivariate trees.
\textcite{10.1609/aaai.v33i01.33011624} propose a binary linear
formulation in which the problem size is largely independent of the
size of the training data.
\rev{Other recent MIP approaches for OCTs have been proposed by
  \textcite{Liu-et-al:2024,Aghaei-et-al:2025,Ales-et-al:2024}.}
MIP approaches that consider support vector machines
\parencite{cortes1995support} to split the tree also have been
explored, as can be seen in
\textcite{SVMDT2,maximumdist,donofrio2023margin}.

Besides the MIP models to solve an OCT, \textcite{BLANQUERO2021105281}
proposed a nonlinear optimization approach to compute an optimal
``randomized'' classification tree.
In their approach, each data point is assigned to a class only with a
given probability.
The model uses \rev{multivariate} cuts with the goal of utilizing fewer predictor
variables in the splits of the tree. \rev{This builds upon the work of
  \textcite{BLANQUERO2020255}, which introduced the use of oblique
  cuts and focused on achieving both local and global sparsity. By
  employing regularization with polyhedral norms, this approach
  specifically aims to utilize fewer predictor variables in the splits
  of the tree, a feature that is also integrated into their subsequent
  randomized model.}

All the strategies presented so far exclusively focus on labeled
data.
However, acquiring labels for every unit of interest can be
expensive---in particular if \rev{face-to-face} surveys are used to
obtain the labels. \rev{These surveys rise in cost more or less
  linearly with the number of surveyed units. Hence, if the
  non-labeled data is available,} it would be beneficial to train the
decision tree on only partly labeled data.
This yields a semi-supervised learning
setting \parencite{semisupervised}. \rev{The most common approach to
  tackle the semi-supervised learning problem is to learn the
  structure of the feature space. As the features are taken to be
  available for the labeled and unlabeled units, their structure can
  be estimated in a more stable way; see, e.g., \textcite{belkin2006manifold,
    kemp2003semi}. \textcite{zhu2003semi} extend this idea, by using
  propagation methods, to fill up data-gaps in the feature space, and
  \textcite{blanco2023multiclass} show how to use the pre-learned
  structure of the feature space for multiclass classification.}
Algorithms for semi-supervised learning have already been proposed for
\rev{support vector machines (SVMs)}
\rev{\parencite{demiriz2000optimization,Chapelle,LAPSVM2}}, neural
networks \parencite{SemiSupNN,NeuralSemiSUpervised,NeuralSEMISUP}, and
logistic regression \parencite{LRSemi,SemiSupRL}.

In the field of semi-supervised decision trees,
\textcite{Semisupervised3} splits internal nodes by
utilizing the structural characteristics of the data for subspace
partition.
Moreover, \textcite{Semisupervised2} consider minimizing a local impurity
function and \textcite{SemisupervisedSF} propose self-training as base
learners.
Recently, \textcite{SemiSuptogrowing} consider a maximum-mean
discrepancy to  estimate the class ratio at every splitting rule
in a univariate decision tree.
Furthermore, \textcite{SemiSupervised4} present a graph Laplacian
approach to deal with unlabeled data.
Hence, although they present different ideas, none of the approaches
considers globally solving a single optimization problem.

Moreover, in many cases, external sources provide information
about the total amount of elements in each class within a
population.
For example, in some businesses, the number of positive labels might
be available, but the identification of which customer has a positive
or negative label is unknown.
An intuitive example is a supermarket for which the amount of cash
payments is known.
However, this information is not attributable to the individual
customers ex-post.
Another example is population surveys, where statistics agencies can
provide how many people are employed.
\rev{In the context of bank lending, the bank can also determine how
  many people will receive the loan and how many will not, providing
  cardinality information. This cardinality information can also be
  obtained from the class distribution of the sample, as it is done in
  \textcite{BURGARD2026107337} for health data generation.}
\rev{Although the available data may be provided by external sources,
  it can come from non-probability samples in which the data
  collection mechanism is largely unknown, creating the risk of a
  biased sample. In such cases, having additional information on the
  total number of points in each class can be particularly important
  because standard methods can struggle to correct for sampling bias.}

For logistic regression, the idea of aggregating this extra
information is proposed by \textcite{Burgard2021CSDA}, who develop
a cardinality-constrained multinomial logit model.
In the SVM setting, \textcite{constrainedSVM} present a mixed-integer
quadratic optimization model and iterative clustering techniques to
tackle cardinality constraints for each class.
Our contribution \rev{in this paper} is to propose to add this
aggregated additional information to a multivariate OCT model by
imposing a cardinality constraint on the predicted labels for the
unlabeled data.
\rev{To this end, we consider a special case of the broad field of
  semi-supervised learning that is mainly focusing on the given
  cardinality information.}

We develop a big-$M$-based MILP to solve the semi-supervised optimal
classification tree (S$^2$OCT) problem that deals with the cardinality
constraint for the unlabeled data.
The cardinality constraint helps to account for biased samples since
the number of predictions in each class on the population is bounded
by the constraint.
This paper is organized as follows.
In Section~\ref{sec:chapter-1} we present preliminary concepts and our
optimization problem.
Afterward, the big-$M$-based MILP formulation is presented in
Section~\ref{sec:chapter-2}.
In Section~\ref{sec:numerical-results}, numerical results are reported
and discussed and we conclude in Section~\ref{sec:conclusion}.


\section{Preliminary Concepts}
\label{sec:chapter-1}

Let $\rev{X = [X_l, X_u]} \in \R^{p \times N}$ be the data matrix with
labeled data $X_l = [x^1, \dotsc, x^{n}]$ and unlabeled data $X_u =
[x^{n+1}, \dotsc, x^N]$.
Hence, we observe $x^i \in \mathbb{R}^p$ for all $i \in [1,N]
\define \set{1,\dots, N}$.
We set $m \define N -n $, $\mathcal{U}\define [n+1,\dots, N]$, and
$c \in \set{\mathcal{A},\mathcal{B}}^{n}$ as the vector of class
labels for the labeled data.

In a multivariate optimal classification tree, each decision consists
of a linear combination of the $p$ components of a point $x^i$.
Considering only the labeled data, the goal is to split the feature
space into distinct regions to correctly classify each point.
However, in many applications, aggregated information on the labels
is available, e.g., from census data.
\rev{In the following, we know the total number~$\lambda \in \N$ of
  unlabeled points that belong to the class $\mathcal{A}$.
  Based on that, we adapt the idea of optimal classification trees such
  that we can use the unlabeled data as well as $\lambda$ as additional
  information.}

Given a depth $D \in \mathbb{N}$, a classification tree has
$2^{D+1}-1$ nodes, categorized in two types; the branch nodes
$\mathcal{T}_B = [1, 2^D - 1]$ and the leaf nodes $\mathcal{T}_L =
[2^D, 2^{D+1}-1]$.
Each branch node $b \in \mathcal{T}_B$ provides a hyperplane
parameterized by $(\omega^b, \gamma_b)$ that splits the feature
space into half-spaces.
As suggested by \textcite{Bertsimas2017},  if a point $x^i$, $i \in
[1,N]$, satisfies $(\omega^b)^\top x^i - \gamma_b \leq 0$,  then
$x^i$ follows the left branch of the node $b$.
If $(\omega^b)^\top x^i - \gamma_b > 0$ holds, the point $x^i$ follows
the right branch of the node $b$.
For an optimization formulation, the strict inequality needs
to be re-written as $(\omega^b)^\top x^i - \gamma_b \geq \varepsilon$
for a sufficiently small $\varepsilon > 0$.
However, a very small value of $\varepsilon$ might lead to numerical
instabilities.
To avoid this, we replace $(\omega^b)^\top x^i - \gamma_b \leq 0$ by
$(\omega^b)^\top x^i - \gamma_b \leq -1$, and $(\omega^b)^\top x^i -
\gamma_b \geq \varepsilon$ by $(\omega^b)^\top x^i - \gamma_b \geq
1$.
\revTwo{Note that the latter is always possible if no data point actually
  lies on a hyperplane, which we will later penalize; see
  Definition~\ref{def:leaf-error} and the objective function
  in~\eqref{eq:minimumerrorob}.}

In each leaf node, $t \in \mathcal{T}_L$, all points $x^i$, $i \in
[1,N]$, are classified as $\mathcal{A}$ or $\mathcal{B}$.
In a simple example with $D=2$, the classification tree has
$\mathcal{T}_B = [1,3]$ and $\mathcal{T}_L = [4,7]$; see
Figure~\ref{fig:DecisionTree}.
Regarding the classification at the leaf nodes, let
$\mathcal{T}_L^{\mathcal{A}} = \defset{t \in \mathcal{T}_L}{t \text{
    is even}}$ be the set of leaf nodes that are classified as
$\mathcal{A}$ and $\mathcal{T}_L^{\mathcal{B}}
= \defset{t \in \mathcal{T}_L}{t \text{ is odd}}$ be the set of leaf
nodes
that are classified as $\mathcal{B}$.
For the classification tree with $D=2$, see
Figure~\ref{fig:DecisionTree} again,
$\mathcal{T}_L^{\mathcal{A}} = \set{4,6}$ and
$\mathcal{T}_L^{\mathcal{B}} = \set{5,7}$ holds.
\begin{figure}
  \centering
  \begin{forest}
  for tree={
    draw, 
    edge={->}, 
    l sep=25pt, 
    s sep=50pt 
  },
  decision/.style={
    circle, 
    inner sep=7pt, 
    minimum size=13pt 
  },
  result/.style={
    rectangle, 
    inner sep=10pt, 
    minimum size=13pt 
  }
  [1, decision
  [2, decision, edge label={node[midway, anchor=east] {\footnotesize  {  $(\omega^1)^\top x^i - \gamma_1 \leq -1 \;$} }}
  [4, result, edge label={node[midway, left] {\footnotesize {  $(\omega^2)^\top x^i - \gamma_2 \leq -1$}}},  label={ below: $\mathcal{A}$}]
  [5, result, edge label={node[midway, right] {\footnotesize {  $ \geq 1$}}}, label={ below: $\mathcal{B}$}]
  ]
  [3, decision, edge label={node[midway, right] {\footnotesize $\; (\omega^1)^\top x^i - \gamma_1 \geq 1$}}
  [6, result, edge label={node[midway, left] {\footnotesize {  $(\omega^3)^\top x^i - \gamma_3 \leq -1$}}}, label={ below: $\mathcal{A}$}]
  [7, result, edge label={node[midway, right] {\footnotesize {  $ \geq 1$}}}, label={ below: $\mathcal{B}$}]
  ]
  ]
\end{forest}%
%
%
  \caption{A classification tree with depth~$D=2$}
  \label{fig:DecisionTree}
\end{figure}

For each leaf node $t \in \mathcal{T}_L$, define $\mathcal{N}_R(t)$ as
the index set of the branch nodes in which the right or ``greater
than'' branch is traversed to reach leaf $t$.
Moreover, we define $\mathcal{N}_L(t)$ as the index set of the branch
nodes in which the left or ``less than'' branch is traversed to reach
leaf $t$.
For the classification tree in Figure~\ref{fig:DecisionTree} we have
\begin{align*}
  \mathcal{N}_L(4) = \{1,2\},
  \quad  \mathcal{N}_R(4) = \emptyset,
  & \quad  \mathcal{N}_L(5) = \{1\},
    \quad  \mathcal{N}_R(5) = \{2\},
  \\
  \mathcal{N}_L(6) = \{3\},
  \quad  \mathcal{N}_R(6) = \{1\},
  & \quad \mathcal{N}_L(7) = \emptyset,
    \quad \mathcal{N}_R(7) = \{1,3\}.
\end{align*}
\rev{Having introduced the tree structure, we can now examine its
  geometric interpretation. Figure~\ref{fig:draws} shows a
  2-dimensional dataset and the partition created by the tree. Each
  split divides the feature space into regions, illustrating how the
  tree forms decision boundaries in the input space.}
\begin{figure}
  \centering
  \includegraphics[width=0.75\textwidth]{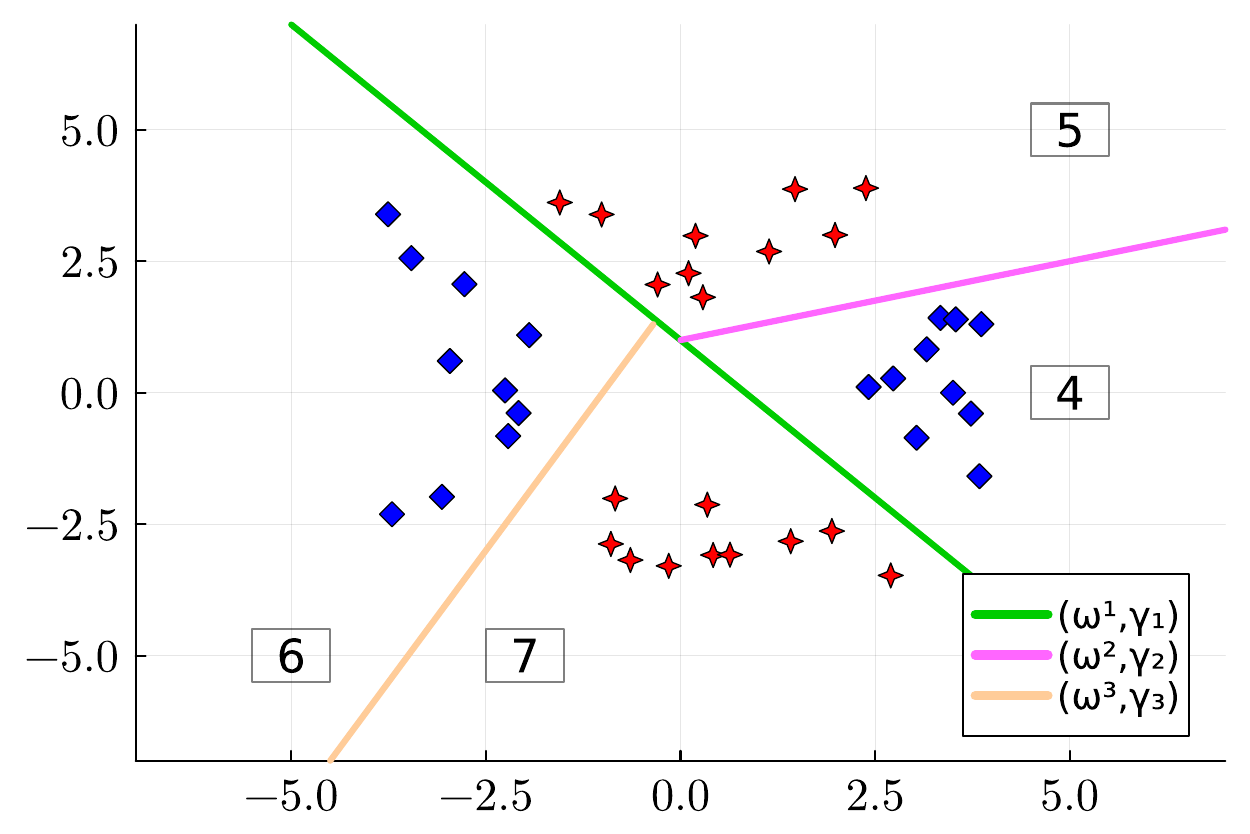}
  \caption{\rev{A 2-dimensional example and the hyperplanes produced by a tree-based partitioning with $D=2$.}}
  \label{fig:draws}
\end{figure}
Given a fixed depth $D$, a  point $x^i$, $i \in [1,n]$, with $c_i \in
\set{\mathcal{A},\mathcal{B}}$, is correctly classified if all
inequalities from the root to the leaf node are satisfied for some
leaf node~$t \in \mathcal{T}_L^{c_i}$.
Hence, a point~$x^i$ is correctly classified if
\begin{equation}
  \label{eq:classifyA}
  \bigvee_{t\in \mathcal{T}_L^{c_i}} \left \{ \left \{ \bigwedge_{b
        \in \mathcal{N}_R(t)} \left  [\left(\omega^b\right)^\top x^i
        -  \gamma_b \geq 1 \right] \right\} \bigwedge \left \{
      \bigwedge_{b \in \mathcal{N}_L(t)  }
      \left[\left(\omega^b\right)^\top x^i -  \gamma_b \leq -1
      \right]\right\}  \right\}.
\end{equation}
is satisfied.

In our running example with $D=2$, a point~$x^i$ with $c_i
=\mathcal{A}$ is correctly classified if
\begin{align*}
  & \left\{\left  [\left(\omega^1\right)^\top x^i -  \gamma_1 \leq - 1
    \right] \wedge \left  [\ \left(\omega^2\right)^\top x^i -  \gamma_2
    \leq - 1  \right] \right\}
  \\
  \vee
  & \left\{ \left
    [\left(\omega^1\right)^\top x^i -  \gamma_1 \geq 1  \right] \wedge
    \left  [\left(\omega^3\right)^\top x^i -  \gamma_3 \leq - 1
    \right]\right\}
\end{align*}
holds.

An unlabeled point~$x^i$, $i \in [n +1, N]$, is classified as
$\mathcal{A}$ if
\begin{equation}
  \label{eq:classifyXasA}
  \bigvee_{t\in \mathcal{T}_L^{\mathcal{A}} } \left \{ \left \{
      \bigwedge_{b \in \mathcal{N}_R(t)} \left[ \left(\omega^b\right)^\top
        x^i -  \gamma_b \geq 1 \right] \right\} \bigwedge \left \{ \bigwedge_{b
        \in \mathcal{N}_L(t)} \left[ \left(\omega^b\right)^\top x^i -
        \gamma_b \leq -1 \right] \right\}  \right\}
\end{equation}
is true.
Thus, our goal is to find $\omega$ and $\gamma$ such that
Expression~\eqref{eq:classifyA} is satisfied for all labeled
data points~$x^i$, $i \in [1,n]$, and such that the number of
unlabeled points~$x^i$, $i \in [n +1 , N]$, that satisfy
Expression~\eqref{eq:classifyXasA} is as close as possible to
$\lambda$.

For doing so, we first need to define suitable error measures that
then have to be minimized.
For this purpose, we define the branch and leaf error according to the
decision error and leaf error proposed by \textcite{bennett1996optimal}.
The first error is related to branch nodes.
For each labeled point, at each branch node, we consider the
inequality that must be satisfied and then measure by how much it is
violated.

\begin{defi}[Branch Error]
  Given a labeled point $x^i$, $i \in [1,n]$, in any branch node~$b
  \in\mathcal{T}_B$, the \emph{branch errors}~$\left(y_b^R\right)_i$
  and~$\left(y_b^L\right)_i $ are defined by
  \begin{align*}
    \left(y_b^R\right)_i
    \define \left[-\left(\omega^b\right)^\top x^i +  \gamma_b +1
    \right]^{+}
    \quad\text{and}\quad
    \left(y_b^L\right)_i
    \define \left[\left(\omega^b\right)^\top x^i -  \gamma_b +1
    \right]^{+}
  \end{align*}
  with $[v]^+ \define \max \set{0, v}$.
\end{defi}

The definition above can be interpreted as follows. If the point $x^i$
satisfies $(\omega^b)^\top x^i -  \gamma_b \geq 1$, and therefore
follows the right branch in some node~$b \in \mathcal{T}_B$, it holds
$(y_b^R)_i = 0$ and $(y_b^L)_i > 0$.
However, if the point follows the left branch in some node~$b \in
\mathcal{T}_B$, i.e., if $(\omega^b)^\top x^i - \gamma_b \leq -1$
holds, we obtain $(y_b^R)_i > 0$ and $(y_b^L)_i = 0$.

The next definition represents the error in each leaf node $t$. For
each labeled point, we sum over the branch errors along the path from
the root to the leaf node~$t$.

\begin{defi}[Leaf Error]
  \label{def:leaf-error}
  The \emph{leaf error} of a labeled point $x^i$, $i \in [1,n]$, at a
  leaf node~$t \in \mathcal{T}_L$ is defined by
  \begin{equation}
    \label{eq:LEab}
    \LE(x^i,t) \define
    \sum_{b \in \mathcal{N}_R(t)} \left(y_b^R\right)_i
    + \sum_{b \in \mathcal{N}_L(t)} \left(y_b^L\right)_i.
  \end{equation}
\end{defi}

Note that $\LE(x^i,t)$ is a linear expression and for each labeled
point~$x^i$, \mbox{$i \in [1,n]$}, Expression~\eqref{eq:classifyA} is
satisfied if $\LE(x^i,t) = 0$ for some $t \in \mathcal{T}_L^{c_i}$.
Additionally, $\LE(x^i,t) \geq 0$ holds for all
$t \in \mathcal{T}_{L}$.
Hence, each labeled point $x^i$ is correctly classified if the minimum
value of all leaf errors is zero, i.e., if
\begin{equation*}
  \min_{t \in c_i} \Set{\LE(x^i,t)} = 0
\end{equation*}
holds.
Besides that, we want to classify $\lambda$ unlabeled points as
$\mathcal{A}$, which means $\lambda$ unlabeled points must end up in
some $t \in \mathcal{T}_L^{\mathcal{A}}$.

To sum up, given the data matrix $X \in \mathbb{R}^{p \times N}$ and
$\lambda \in \mathbb{N}$ as well as some $s > 0$, our goal is
to find optimal parameters $\omega \in \mathbb{R}^{p\times 2^D-1}$,
$\gamma \in \mathbb{R}^{2^D-1}$ as well as $y^R, y^L \in
\mathbb{R}^{2^D-1 \times n}$ and $\xi \in \mathbb{R}$ that solve the
optimization problem
\begin{varsubequations}{P1}
  \label{eq:minimumerror}
  \begin{align}
    \min_{\omega,\gamma,y^R, y^L,\xi} \quad
    & \sum_{i=1}^{n} \min_{t\in \mathcal{T}_L^{c_i}} \{\LE(x^i,t)\}  +
      C\xi\label{eq:minimumerrorob}
    \\
    \st  \quad
    &\left(y_b^R\right)_i \geq -\left(\omega^b\right)^\top x^i +  \gamma_b +1, \quad
      b \in \mathcal{T}_B, \quad i \in [1,n],
      \label{constraint1}\\
    &\left(y_b^L\right)_i \geq \left(\omega^b\right)^\top x^i  -\gamma_b +1, \quad
      b \in \mathcal{T}_B, \quad i \in [1,n],
      \label{constraint2} \\
    & -s \leq \omega_j^b \leq s, \quad b \in\mathcal{T}_B, \quad j \in
      [1,p],
      \label{boundomega}\\
    & \left(y_b^R\right)_i, \left(y_b^L\right)_i \geq 0, \quad  b
      \in\mathcal{T}_B, \quad i \in [1,n],
      \label{posity}\\
    & \lambda - \xi \leq  \sum_{i=n+1}^{N}
      \sum_{t\in\mathcal{T}_L^{\mathcal{A}} }\left(\psi(x^i,t)\right)
      \leq \lambda + \xi,
      \label{constraintxi} \\
    &\xi \geq 0
      \label{constraintend}
  \end{align}
\end{varsubequations}
where  $\LE(x^i,t)$ is defined in \eqref{eq:LEab} and
\begin{equation}\label{psi}
  \psi(x^i,t) =
  \begin{cases}
    1, & \text{if } x^i \text{ ends in the leaf node } t,\\
    0, & \text{otherwise}.
  \end{cases}
\end{equation}
Note that the objective function in~\eqref{eq:minimumerrorob} models a
compromise between minimizing the classification error for the labeled
and unlabeled data.
The penalty parameter~$C>0$ aims to control the importance of the
slack variable~$\xi$.
Constraints~\eqref{constraint1} and~\eqref{constraint2} enforce on
which branch (left or right) the labeled point $x^i$ should traverse
in branch node~$b$.
Constraint~\eqref{boundomega} defines the domain of each
variable~$\omega^b$.
This constraint is not necessary for the correctness of the model but
will serve as a big-$M$-type parameter that is later used for deriving
certain bounds; see Section~\ref{sec:chapter-2}.
Constraint~\eqref{constraintxi} ensures that the number of unlabeled data
classified as $\mathcal{A}$ is as close to $\lambda$ as possible.

The functions $\min\defset{\LE(x^i,t)}{t\in \mathcal{T}_L^{c_i}}$ in the
objective function~\eqref{eq:minimumerrorob} and $\psi(x^i,t)$ in
Constraint~\eqref{constraintend} are discontinuous, which means that
Problem~\eqref{eq:minimumerror} cannot be solved easily by standard
solvers as such.
Hence, we will present a mixed-integer linear programming (MILP)
formulation using binary variables to re-model the objective
function~\eqref{eq:minimumerrorob} and to count the classification of
unlabeled points.


\section{The MILP Model}
\label{sec:chapter-2}

We start the development of a MILP model by using classic
SOS1-techniques \parencite{SOS1} and McCormick
inequalities \parencite{Mccormick} to re-phrase a $\min \min $ problem
as an MILP formulation.

\begin{lem}
  \label{lemma2}
  Consider a set of continuous functions $f_k : \mathbb{R}^p \to
  \mathbb{R}$, $k \in [1,d]$, for some $d\in \mathbb{N}$, and let $\Omega
  \subseteq \mathbb{R}^p$ be given.
  Suppose further that there exist values $u_k>0$ such that $$0
  \leq f_k(x)\leq u_k$$ holds for all $x \in  \Omega$ and $k \in
  [1,d]$.
  Then, $x^*\in \mathbb{R}^p$ is a solution to the problem
  \begin{subequations} \label{problemmain}
    \begin{align}
      \min_{x} \quad
      & \min_{k \in [1,d]} f_k(x)
        \label{eq:problemmain1}\\
      \st \quad
      & x \in \Omega
        \label{eq:problemmain2}
    \end{align}
  \end{subequations}
  if and only if there exist $\alpha^*, \beta^* \in \mathbb{R}^d$ such
  that $(x^*, \alpha^*, \beta^*)$ is a solution to the problem
  \begin{subequations}
    \label{problembeta}
    \begin{align}
      \min_{ x,\alpha,\beta }
      & \quad \sum_{k=1}^d \beta_k
        \label{eq:beta1} \\
      \st
      & \quad x \in  \Omega,
        \label{eq:beta2}  \\
      & \quad  \sum_{k =1}^d \alpha_k = 1,
        \label{eq:beta3} \\
      &  \quad    \alpha_k \in \{0,1\}, \quad k \in[1,d],
        \label{eq:beta4}     \\
      &  \quad    \beta_k\leq u_k\alpha_k, \quad k \in[1,d],
        \label{eq:beta5}   \\
      &  \quad    \beta_k\leq f_k(x), \quad k \in[1,d], \quad x \in
        \Omega ,
        \label{eq:beta6}   \\
      &  \quad   \beta_k\geq f_k(x) - u_k(1-\alpha_k), \quad k
        \in[1,d], \quad x \in  \Omega,
        \label{eq:beta7}   \\
      &  \quad    \beta_k\geq 0, \quad k \in[1,d].
        \label{eq:beta8}
    \end{align}
  \end{subequations}
\end{lem}
\begin{proof}
  By introducing the binary variables~$\alpha_k$, we obtain that $x^*$
  is a solution to Problem~\eqref{problemmain} if and only if $(x^*,
  \alpha^*)$ is a solution of the problem
  \begin{align*}
    \min_{ x,  \alpha  } \quad
    & \sum_{k =1}^d  \alpha_k f_k(x)
    \\
    \st \quad
    &  \eqref{eq:beta2}\text{--}\eqref{eq:beta4}.
  \end{align*}

  Besides that, for any $(x^*, \alpha^*)$ solution of the problem
  above, if $\alpha_k^* = 0$ holds for some $k\in [1,d]$, Constraints
  \eqref{eq:beta5} and \eqref{eq:beta8} enforce that $\beta^*_k = 0 =
  \alpha_k f_k(x^*)$. On the other hand, if $\alpha_k^*=1$ holds, by
  Constraints~\eqref{eq:beta6} and~\eqref{eq:beta7}, we obtain
  $\beta_k^* = \alpha^*_k f_k(x^*)$.

  Hence, $x^*$ is a solution to Problem~\eqref{problemmain} if and
  only if $(x^*, \alpha^*,\beta^*)$ exists such that
  \begin{equation*}
    \beta^*_k = \alpha^*_k f_k(x^*), \quad k \in [1,d],
  \end{equation*}
  is a solution to Problem~\eqref{problembeta}.
\end{proof}

To apply the previous lemma to Problem \eqref{eq:minimumerror} it is
necessary that $\LE(x^i, t)$ has lower and upper bounds, which are
given in following proposition.
Here and in what follows, $\Vert \cdot \Vert$ denotes the Euclidean
norm.

\begin{prop}
  \label{boundhiperprop}
  Given $[x^1, \dots, x^N]$ with $x^i \in \mathbb{R}^p$ for all $i
  \in [1,N]$ and $s > 0$, every optimal solution ($\omega,
  \gamma,y^{R},y^{L},\xi)$ to Problem~\eqref{eq:minimumerror}
  satisfies
  \begin{equation*}\label{boundhiper*}
    \vert (\omega^b)^\top x^i - \gamma_b \vert \leq \eta s \sqrt{p},
    \quad i \in [1,N], \quad b \in  \mathcal{T}_B,
  \end{equation*}
  and
  \begin{equation*}
    (y_b^R)_i \leq \eta s \sqrt{p} +1, \quad (y_b^L)_i \leq \eta s
    \sqrt{p} +1, \quad  i \in [1,n], \quad b \in  \mathcal{T}_B,
  \end{equation*}
  where
  $\eta \define \max_{i,k \in [1,N]} \Vert x^i -x^k \Vert$.
  Moreover,
  \begin{equation*}
    0 \leq \LE(x^i, t) \leq  B(s), \quad t \in i \in [1,n], \quad t \in  \mathcal{T}_L,
  \end{equation*}
  holds with
  \begin{equation}
    \label{boundLE}
    B(s) \define D(\eta s \sqrt{p} +1).
  \end{equation}
\end{prop}
\begin{proof} Let $\mathcal{H}$ be the set of hyperplanes that separates $[x^1,
  \dots, x^N] $ into two sets, and let the hyperplane $(\omega^b,
  \gamma_b) \in \mathcal{H}$ be the hyperplane with the minimal
  maximum distance to any point in $X$. Then, according to
  \textcite{maximumdist}, the maximum distance from a point $x^i,
  i~\in~[1,N]$, to $(\omega^b, \gamma_b)$  is $\eta$, i.e.,
  \begin{equation*}
    \dfrac{\vert (\omega^b)^\top x^i - \gamma_b \vert}{\Vert
      \omega^b\Vert } \leq \eta.
  \end{equation*}
  Moreover, Constraint~\eqref{boundomega} enforces that $\Vert
  \omega^b\Vert \leq s\sqrt{p}$.
  Hence,
  \begin{equation*}
    \vert (\omega^b)^\top x^i - \gamma_b \vert \leq \eta s \sqrt{p},
    \quad i \in [1,N], \quad b \in  \mathcal{T}_B,
  \end{equation*}
  holds and
  \begin{equation*}
    -\left(\omega^b\right)^\top x^i +  \gamma_b +1 \leq  \eta s
    \sqrt{p} +1 \quad \text{ as well as } \quad
    \left(\omega^b\right)^\top x^i  -\gamma_b +1 \leq  \eta s \sqrt{p}
    +1
  \end{equation*}
  are satisfied for all $b \in \mathcal{T}_B$ and $i \in [1,n]$.

  Note further that Problem~\eqref{eq:minimumerror} is a minimization
  problem and $\LE(x^i, t)$ is a sum of the $D$ non-negatives
  terms~$y_b^R$ and~$y_b^L$.
  Thus, Constraints~\eqref{constraint1} and~\eqref{constraint2} imply
  that
  \begin{equation*}
    (y_b^R)_i \leq \eta s \sqrt{p} +1, \quad (y_b^L)_i \leq \eta s
    \sqrt{p} +1, \quad  i \in [1,n], \quad b \in  \mathcal{T}_B,
  \end{equation*}
  and
  \begin{equation*}
    0 \leq \LE(x^i, t) \leq  B(s), \quad t \in \mathcal{T}_L,
  \end{equation*}
  holds.
\end{proof}
Note that Constraint~\eqref{boundomega} is decisive for obtaining the
upper bound $B(s)$ in the previous lemma.
Finally, to overcome the discontinuity of the function $\psi(\cdot)$
defined in \eqref{psi}, we also add binary variables and use SOS
techniques again to turn on or off the enforcement of a constraint.
More formal, by introducing the matrices $\beta \in
\mathbb{R}^{n\times 2^{D-1}}$, $\alpha \in \{0,1\}^{n\times 2^{D-1}}$,
$z \in \{0,1\}^{m\times 2^{D}-1}$, and $\delta \in \{0,1\}^{m \times
  2^{D-1}}$ of binary variables, we can reformulate the optimization
problem~\eqref{eq:minimumerror} as follows:
\begingroup
\allowdisplaybreaks
\begin{varsubequations}{P2}
  \label{eq:MIO}
  \begin{align}
    \min_{\omega,\gamma,y^R, y^L, \xi, \alpha, \beta , z,\delta }
    \quad
    & \sum_{i=1}^{n} \sum_{t\in \mathcal{T}^{c_i}_L} \beta_t^i +C\xi \label{eqMIOobj}
    \\   \st  \quad
    & \eqref{constraint1}\text{--}\eqref{posity},
      \nonumber\\
    & \sum_{t\in \mathcal{T}^{c_i}_L} \alpha_t^i\ = 1, \quad i \in
      [1,n],
      \label{constraint2alpha} \\
    & \alpha^i_t \in \{0,1\}, \quad i \in [1,n], \quad t \in
     \mathcal{T}_L^{c_i},
      \label{constraintalpha} \\
    & \beta^i_t \leq \LE(x^i,t) , \quad i \in [1,n], \quad t \in
      \mathcal{T}_L^{c_i},
      \label{boundbeta1}\\
    &  \beta^i_t  \geq  \LE(x^i,t) -B(s)(1-\alpha^i_t)  , \quad i \in
      [1,n],  \; t \in   \mathcal{T}_L^{c_i},
      \label{boundbeta2} \\
    &0 \leq \beta^i_t \leq B(s) \alpha^i_t, \quad i \in [1,n], \quad t
      \in \mathcal{T}_L,
      \label{boundbeta3}\\
    &  (\omega^b)^\top x^i - \gamma_b \leq z^b_i M -1, \quad b \in
      \mathcal{T}_B, \quad  i \in \mathcal{U},
      \label{BIGM1} \\
    &   (\omega^b)^\top x^i - \gamma_b \geq -(1-z^b_i) M +1, \quad b
      \in  \mathcal{T}_B, \quad i \in \mathcal{U},
      \label{BIGM2} \\
    & z^b_i \in \{0,1\}, \quad b \in   \mathcal{T}_B, \quad i \in
      \mathcal{U},
      \label{zbinary}\\
    &  \delta^t_i \leq z^b_i, \quad b\in \mathcal{N}_R(t), \quad t \in
      \mathcal{T}_L^{\mathcal{A}}, \quad i \in \mathcal{U},
      \label{binarydelta1} \\
    &  \delta^t_i \leq - z^b_i + 1, \quad b\in \mathcal{N}_L(t), \quad
      t \in \mathcal{T}_L^{\mathcal{A}}, \quad i \in \mathcal{U},
      \label{binarydelta2}  \\
    &  \delta^t_i \geq   \sum_{b\in \mathcal{N}_R(t)}  z^b_i +
      \sum_{b\in \mathcal{N}_L(t)} (- z^b_i + 1 ) - (D-1), \; t \in
      \mathcal{T}_L^{\mathcal{A}}, \; i \in \mathcal{U},
      \label{binarysum2} \\
    & \delta^t_i \in \{0,1\}, \quad t \in \mathcal{T}_L^{\mathcal{A}},
      \quad i \in \mathcal{U},
      \label{endMIO}\\
    &  \lambda - \xi  \leq   \sum_{i = 1+n}^{N }\sum_{t \in
      \mathcal{T}_L^{\mathcal{A}}} \delta^t_i  \leq \lambda + \xi,
      \label{sumbin}  \\
    &  \xi \geq 0,
  \end{align}
\end{varsubequations}
\endgroup
where $M$ needs to be chosen sufficiently large.
The constraints in~\eqref{constraintalpha} enforce that the minimum
value of $\LE(x^i,t)$ is selected for each $x^i \in [1,n]$ and
Constraints~\eqref{boundbeta1}--\eqref{boundbeta3} ensure
\begin{equation*}
  \beta^i_t = \alpha^i_t \LE(x^i,t),
  \quad i \in [1,n], \quad t \in \mathcal{T}_L.
\end{equation*}
As $z^b_i$ is binary, Constraints \eqref{BIGM1} and \eqref{BIGM2} lead
to
\begin{align*}
  (\omega^b)^\top x^i - \gamma_b \geq 1 \implies z^b_i = 1, \quad b
  \in   \mathcal{T}_B, \quad i \in \mathcal{U},
  \\
  (\omega^b)^\top x^i - \gamma_b \leq -1 \implies z^b_i   = 0, \quad b
  \in   \mathcal{T}_B, \quad i \in \mathcal{U}.
\end{align*}
Furthermore, as $\delta_i^t$ is binary as well, for all $i \in
\mathcal{U}$ and $t \in \mathcal{T}_L^{\mathcal{A}}$,
Constraints~\eqref{binarydelta1}--\eqref{binarysum2} lead to
\begin{equation*}
  \delta_i^t =
  \begin{cases}
    1, & \text{ if }  z_i^b = 1 \text{ for } b \in \mathcal{N}_R(t)
    \text{ and } z_i^b = 0 \text{ for } b \in \mathcal{N}_L(t),
    \\
    0, & \text{ otherwise},
  \end{cases}
\end{equation*}
i.e.,
\begin{equation*}
  \delta_i^t =
  \begin{cases}
    1, & \text{ if }  x^i \text{ ends up in the leaf node } t,
    \\
    0, & \text{ otherwise}.
  \end{cases}
\end{equation*}
Constraint~\eqref{sumbin} ensures that the number of unlabeled data
classified as $\mathcal{A}$ is as close to $\lambda$ as possible.
Reformulation~\eqref{eq:MIO} is an MILP.
We refer to this problem as S$^2$OCT.
As usual in mixed-integer optimization, the big-$M$-value
is crucial as is the choice of~$s$ in Constraint~\eqref{boundomega}.
However, precisely based on $s$ we have an exact value for $M$, as
discussed in the following result.

\begin{prop}\label{prop:BIGM}
  Consider $\eta \define \max_{i,k \in [1,N]} \Vert x^i-x^k \Vert$.
  Any feasible point for Problem~\eqref{eq:MIO} satisfies
  \eqref{BIGM1} and \eqref{BIGM2} for
  $M = \eta s \sqrt{p} +1$.
\end{prop}
\begin{proof}
  Note that if $z_i^b =1$ for some $i \in \mathcal{U}$ and $b \in
  \mathcal{T}_B$, Constraints~\eqref{BIGM1} and~\eqref{BIGM2} imply
  \begin{equation*}
    M-1 \geq (\omega^b)^\top x^i - \gamma_b \geq 1.
  \end{equation*}
  Moreover, since $(\omega^b)^\top x^i - \gamma_b\geq 0$, because of
  Proposition \ref{boundhiperprop}, we get
  \begin{equation*}
    (\omega^b)^\top x^i - \gamma_b  \leq \eta s \sqrt{p}.
  \end{equation*}
  This means that $M = \eta s \sqrt{p} +1 $ does not cut off any
  feasible solution.

  On the other hand, if $z_i^b =0$ holds for some $i \in \mathcal{U}$
  and $b \in \mathcal{T}_B$, due to Constraint~\eqref{BIGM1}
  and~\eqref{BIGM2}, we obtain
  \begin{equation*}
    1 - M \leq (\omega^b)^\top x^i - \gamma_b \leq -1,
  \end{equation*}
  and, similarly, $M = \eta s \sqrt{p} +1$ does not cut of any
  feasible solution as well.
\end{proof}


\section{Numerical Results}
\label{sec:numerical-results}

In this section, we present and discuss our computational results that
exemplify the advantages of considering the known total amount of each
class. We analyze this on different test sets from the literature. The
test sets are discussed in Section~\ref{subsection-test-sets}, while
the computational setup is described in
Section~\ref{subsection-comp-setup}. The evaluation criteria are
depicted in Section~\ref{comparsions-OCT}. Finally, the numerical
results are discussed in Section~\ref{numerical-results}.

\subsection{Tests Sets}
\label{subsection-test-sets}

For the computational analysis of the proposed approach, we
consider the subset of instances presented by \textcite{Olson2017PMLB}
that are applicable to classification problems and that have at most
three classes. Repeated instances are eliminated, and all instances
are reduced to complete cases only. If an instance contains three
classes, we convert them into two classes, such that the class with
label~1 represents the class  $\mathcal{A}$  and the other two classes
represent the class $\mathcal{B}$.
This results in a final test set of 97~instances, as listed in
Table~\ref{table1}.
To avoid numerical instabilities, all data sets are scaled as
follows.
For each coordinate $j \in [1,p]$, we compute
\begin{equation*}
  l_j = \min_{i \in [1,N]}\Set{x^i_j},
  \quad
  u_j = \max_{i \in [1,N]}\Set{x^i_j},
  \quad
  m_j = 0.5 \left(l_j + u_j \right)
\end{equation*}
and shift each coordinate~$j$ of all data points~$x^i$ via $\bar{x}^i_j
= x^i_j - m_j$. Furthermore, if a coordinate~$j$ of the
re-scaled points is still large, i.e.,  if $\tilde{l}_j = l_j - m_j <-
10^{2}$ or $\tilde{u}_j = u_j - m_j > 10^{2}$ holds, it is re-scaled
via
\begin{equation*}
  \tilde{x}^i_j = (\bar{v} - \underline{v} ) \frac{\bar{x}^i_j -
    \tilde{l}_j}{\tilde{u}_j-\tilde{l}_j} +  \bar{v}
\end{equation*}
with $\bar{v} = 10^2$ and $ \underline{v} = -10^{2}$.
The corresponding 10 instances that we re-scale are marked with an
asterisk in Table~\ref{table1}.

\revTwo{Categorical features are treated as numerical
  variables. Unlike linear models that require one-hot encoding,
  tree-based models naturally handle numerical splits and capture
  nonlinear relationships. While this numerical treatment implicitly
  imposes an ordinal structure, tree models can mitigate this
  limitation through recursive re-splitting, provided the tree depth
  is sufficient. For shallow trees, one-hot encoding or category
  relabeling can improve splitting efficiency. However, one-hot
  encoding expands the feature space and increases computational
  overhead. Given that this work focuses primarily on the optimization
  problem rather than feature engineering, we use the simpler
  numerical encoding to maintain computational efficiency.}

\begin{center}
  \begin{longtable}{l c c c }
    \caption{Overview over the entire test set with number of
      points~($N$) and dimension~($p$). \rev{Re-scaled instances are
        marked with an asterisk.}}
    \label{table1}\\
    \toprule
    ID & Instance  &$ N $ &$p$ \\
    \midrule
    1                         & prnn\_synth         & 250          & 2   \\
    $2^*$   &  analcatdata\_asbestos        &  73    &  3   \\
    $3^*$    &  lupus                        &  87    &  3  \\
    4     &  analcatdata\_boxing1         &  120   &  3  \\
    5       &  analcatdata\_boxing2         &  132   &  3   \\
    6        &  haberman     &  289   &  3     \\
    7        &  analcatdata\_happiness       &  60    &  3  \\
    $8^*$     &  analcatdata\_aids            &  50    &  4  \\
    9  &  analcatdata\_lawsuit         &  263   &  4  \\
    10     &  iris       &  147   &  4 \\
    11          &  hayes\_roth        &  93    &  4 \\
    12       &  balance\_scale               &  625   &  4 \\
    13       &  parity5     &  32    &  5  \\
    $14^*$        &  bupa                         &  341   &  5  \\
    15        &  irish                        &  470   &  5 \\
    16         &  phoneme       &  5349  &  5  \\
    17                        &  tae                          &  110   &  5  \\
    18                        &  new\_thyroid                 &  215   &  5  \\
    $19^*$                        &  analcatdata\_bankruptcy      &  50    &  6 \\
    $20^*$                        &  analcatdata\_creditscore     &  100   &  6  \\
    21                        &  mux6       &  64    &  6 \\
    22                        &  monk3                        &  357   &  6 \\
    23                        &  monk1                        &  432   &  6 \\
    24                        &  monk2                        &  432   &  6 \\
    25                        &  appendicitis                 &  106   &  7  \\
    26                        &  prnn\_crabs                  &  200   &  7  \\
    $27^*$ &  penguins                     &  333   &  7  \\
    28 &  postoperative\_patient\_data &  78    &  8  \\
    $29^*$                        &  biomed                       &  209   &  8  \\
    $30^*$                        &  pima                         &  768   &  8  \\
    $31^*$                        &  cars                         &  392   &  8  \\
    32                        &  analcatdata\_japansolvent    &  52    &  9  \\
    33                        &  glass2                       &  162   &  9  \\
    34                        &  breast\_cancer               &  272   &  9  \\
    35                        &  saheart                      &  462   &  9  \\
    36                        &  threeOf9                     &  512   &  9 \\
    37                        &  profb                        &  672   &  9  \\
    38                        &  breast\_w                    &  463   &  9 \\
    39                        &  tic\_tac\_toe                &  958   &  9  \\
    40                        &  xd6                          &  512   &  9  \\
    41                        &  cmc                          &  1425  &  9  \\
    $42$ &  analcatdata\_cyyoung9302     &  92    &  10\\
    $43$ &  analcatdata\_cyyoung8092     &  97    &  10 \\
    44 &  breast                       &  691   &  10 \\
    45                        &  flare                        &  315   &  10 \\
    46                        &  parity5+5                    &  1024  &  10\\
    $47$                        &  magic                        &  18905 &  10 \\
    48                        &  analcatdata\_fraud           &  42    &  11 \\
    $49$                        &  heart\_statlog               &  270   &  13 \\
    50                        & heart\_h                                                                                    & 293                          & 13  \\
    51 &  hungarian                             &  293   &  13   \\
    $52^*$ &  cleve           &  302   &  13   \\
    $53^*$ &  heart\_c         &  302   &  13  \\
    54 &  wine\_recognition           &  178   &  13   \\
    $55^*$ &  australian               &  690   &  14 \\
    $56^*$ &  adult   &  48790 &  14   \\
    $57^*$ &  schizo    &  340   &  14   \\
    $58^*$ &  buggyCrx    &  690   &  15   \\
    59 &  labor        &  57    &  16   \\
    60 &  house\_votes\_84           &  342   &  16   \\
    61 &  hepatitis             &  155   &  19   \\
    $62^*$ &  credit\_g                    &  1000  &  20   \\
    63 &  gametes\_e\_0.1H             &  1599  &  20   \\
    64 &  gametes\_e\_0.4H                &  1600  &  20   \\
    65 &  gametes\_e\_0.2H               &  1600  &  20   \\
    66 &  gametes\_h\_50 &  1592  &  20   \\
    67 &  gametes\_h\_75 &  1599  &  20   \\
    $68^*$ &  churn      &  5000  &  20   \\
    $69^*$ &  ring        &  7400  &  20   \\
    70 &  twonorm         &  7400  &  20   \\
    71 &  waveform\_21        &  5000  &  21   \\
    72 &  ann\_thyroid            &  7129  &  21   \\
    73 &  spect   &  228   &  22   \\
    74 &  horse\_colic            &  357   &  22   \\
    75 &  agaricus\_lepiota         &  8124  &  22   \\
    $76^*$ &  hypothyroid              &  3086  &  25   \\
    $77^*$ &  dis                  &  3711  &  29   \\
    $78^*$ &  allhypo      &  3709  &  29   \\
    $79^*$ &  allbp                                       &  3711  &  29   \\
    $80^*$ &  breast\_cancer\_wisconsin           &  569   &  30   \\
    81 &  backache &  180   &  32   \\
    82 &  ionosphere        &  351   &  34   \\
    83 &  chess        &  3196  &  36   \\
    84 &  waveform\_40    &  5000  &  40   \\
    85 &  connect\_4                   &  67557 &  42   \\
    86 &  spectf      &  267   &  44   \\
    $87^*$ &  tokyo1      &  959   &  44   \\
    88 &  molecular\_biology\_promoters       &  106   &  57   \\
    $89^*$ &  spambase            &  4210  &  57   \\
    90 &  sonar         &  208   &  60   \\
    91 &  splice        &  2903  &  60   \\
    92 &  coil2000       &  8380  &  85   \\
    $93^*$ &  Hill\_Valley\_without\_noise     &  1212  &  100  \\
    $94^*$ &  clean1        &  476   &  168  \\
    $95^*$ &  clean2                  &  6598  &  168  \\
    96 &  dna                         &  3002  &  180  \\
    97 &  gametes\_e\_1000atts       &  1600  &  1000 \\
    \bottomrule
  \end{longtable}
\end{center}


In our computational study, we aim to emphasize the significance of
cardinality constraints, particularly in the context of non-representative
biased samples.  Biased samples are frequently observed in non-probability
surveys, which are surveys where the inclusion process is not
monitored and, hence, the inclusion probabilities are unknown as
well. Therefore, correction methods like inverse inclusion probability
weighting are not applicable. For an understanding of inverse
inclusion probability weighting, see \textcite{Skinner2011} and
references therein.

To simulate this scenario, we create 5~biased samples with
$\SI{10}{\percent}$ of the data being labeled for each instance.
In contrast to a simple random sample,  where each point has an equal
probability of being chosen as labeled data, in the biased sample the
labeled data are chosen with probability $\SI{85}{\percent}$ for being
on class $\mathcal{A}$.
\rev{Hence, we obtain the labeled data, the unlabeled data, and have additionally the known total of class $\mathcal{A}$ in the full population.}
Then, for each instance, with a
time limit of \SI{7200}{\second} each, we apply the methods listed in
Section~\ref{subsection-comp-setup}.

Additionally, in Appendix~\ref{sec:num-results-simple-sample}, we
provide the results under simple random sampling, which produces
unbiased samples. We see that the results from the proposed methods
are similar to OCT-H \parencite{Bertsimas2017} in this
setting. Despite the extra computational costs, we do not observe
any drawbacks by using S$^2$OCT in more simple survey designs.

\subsection{Computational Setup}
\label{subsection-comp-setup}

For each one of the 485~instances described in
Section~\ref{subsection-test-sets}, the following approaches are
compared:
\begin{enumerate}
\item[(a)] OCT-H as described in \textcite{Bertsimas2017}, where only
  labeled  data are considered.
\item[(b)] S$^2$OCT as given in Problem \eqref{eq:MIO} with $B(s)$ as
  in Expression \eqref{boundLE} and $M$ as given in
  Proposition~\ref{prop:BIGM}.
\end{enumerate}
Our comparison has been implemented in \codename{Julia}~1.8.5 and we
use \codename{Gurobi}~9.5.2 and
\codename{JuMP} \parencite{DunningHuchetteLubin2017} to solve  OCT-H
as well as  Problem \eqref{eq:MIO}. All computations were executed on
the high-performance cluster ``Elwetritsch'', which is part of the
``Alliance of High-Performance Computing Rheinland-Pfalz'' (AHRP). We
used a single Intel XEON SP 6126 core with \SI{2.6}{\giga\hertz} and
\SI{64}{\giga\byte}~RAM. \rev{All code that was used for this research is made publicly available at
\url{https://github.com/mariaepinheiro/S2OCT}.}

To give the same importance to labeled and unlabeled data, \rev{based
  on preliminary numerical tests,} we set the
parameter $C=1$ in S$^2$OCT. Furthermore, we set the complexity
parameter $\alpha = 0$ in OCT-H. Also, as required by OCT-H, all
points $x^i$ belong to $[0,1]^d$.  For this to hold, we re-scaled the
data as discussed in Section~\ref{subsection-test-sets}, with the
difference that $\tilde{l}_j < 0$ and $\tilde{u}_j > 1$.

Proposition~\ref{prop:BIGM} establishes a relationship between $s$ and
$M$. To keep $M$ at least 500 and $s$ sufficiently large, from
preliminary numerical tests we set, in S$^2$OCT,
\begin{equation*}
  s =
  \begin{cases}
    \max\{10,499/(\eta\sqrt{d})\},  & \text{if } N \in  [1,650),\\
    \max\{20,499/(\eta\sqrt{d})\},  &  N \in  [650,1500).\\
    \max\{40,499/(\eta\sqrt{d})\},  & \text{otherwise}.
  \end{cases}
\end{equation*}
By default, MIP solvers such as \codename{Gurobi} aim to achieve a balance between
exploring new feasible solutions and verifying the optimality of the
current solution. In preliminary numerical tests, we observed that the
solver required significant time to find feasible solutions. Therefore,
we selected \codename{Gurobi}'s parameter \textsf{MIPFocus} = 1, i.e., the solver
focuses more on finding feasible solutions.
Moreover, $D$ in OCT-H and in S$^2$OCT are fixed as
\begin{equation*}
  D=
  \begin{cases}
    2,  & \text{if } N \in  [1,1000),\\
    3,  & \text{otherwise}.
  \end{cases}
\end{equation*}

\subsection{Evaluation Criteria}
\label{comparsions-OCT}

The first evaluation criterion is the run time of \mbox{OCT-H} and S$^2$OCT.
To compare run times, we use empirical cumulative distribution
functions (ECDFs).
Specifically, for $S$ being a set of solvers (or approaches as above)
and for $\bar{P}$ being a set of problems, we denote by $t_{\bar{p},s}
\geq 0$ the run time of the approach~$s \in S$ applied to the
problem~$\bar{p} \in \bar{P}$ in seconds. If $t_{\bar{p},s} > 7200$,
we consider problem~$\bar{p}$ as not being solved by approach~$s$.
With these notations, the performance profile of approach~$s$ is the
graph of the function~$\gamma_s : [0, \infty) \to [0,1]$ given by
\begin{equation*}
  \gamma_s(\sigma) = \frac{1}{\vert \bar{P}
    \vert}\big\vert\left\{\bar{p} \in \bar{P}:
    t_{\bar{p},s} \leq \sigma \right\}\big \vert.
\end{equation*}
Furthermore, since the true labels of all points are known in the
simulation, we categorize them into four distinct categories: true
positive (TP) or true negative (TN) if the point is classified
correctly in classes $\mathcal{A}$ or $\mathcal{B}$, respectively, as
well as false positive (FP) if the point is misclassified in the
class $\mathcal{A}$  and as false negative (FN) if the point is
misclassified in the class $\mathcal{B}$.
Using this information, we calculate two classification metrics.
The first one is accuracy ($\AC$).
It measures the proportion of correctly classified points and is given
by
\begin{equation}
  \label{AC}
  \AC \define \frac{\TP + \TN}{\TP + \TN + \FP + \FN} \in [0,1].
\end{equation}
Observe that for $\AC$ the greater the value, the better the
classification.
The second metric is Matthews correlation coefficient ($\MCC$), \rev{ which is applicable also to unbalanced settings \parencite{chicco2023matthews}}. It
measures the  correlation coefficient between the observed and
predicted  classifications and is computed by
\begin{equation}
  \label{MCC}
  \MCC \define \frac{\TP \times \TN - \FP \times \FN }{\sqrt{(\TP +
      \FP)( \TP + \FN)(\TN + \FP)(\TN + \FN)}} \in [-1,1].
\end{equation}
As for accuracy, the higher the $\MCC$, the better the classification.  The
main question is the following: For a specific instance, does S$^2$OCT or OCT-H have
a higher accuracy and $\MCC$? Hence, we compute the instance-wise
difference of the accuracy and $\MCC$ according to
\begin{equation}
  \label{comparOCTH}
  \overline{\AC} \define \AC_{\SOCT}-\AC_{\OCTH}
  \quad
   \overline{\MCC} \define \MCC_{\SOCT}-\MCC_{\OCTH},
\end{equation}
where $\AC_{\OCTH}$ and $\AC_{\SOCT}$ are computed as in~\eqref{AC},
and $\MCC_{\OCTH}$ and $\MCC_{\SOCT}$ as in~\eqref{MCC}.
To keep the numerical results section concise, we report on precision
and recall in Appendix~\ref{sec:furth-numer-results}.

\subsection{Numerical Results}
\label{numerical-results}

\subsubsection{Run Time}
\label{sec:run-time}

The number of continuous and binary variables is an important property
for comparing different approaches. For this purpose, we
computed the number of these variables for all instances presented
in Section~\ref{subsection-test-sets}. Table~\ref{tablevariables}
provides a quantile analysis of these quantities.

Observe that S$^2$OCT has more variables than OCT-H. Therefore, it can be
expected that OCT-H solves more instances than S$^2$OCT within the
time limit.  However, note that the two approaches are designed for
different purposes.  Our approach considers labeled and unlabeled
points together with the respective cardinality constraint, while
OCT-H only deals with labeled points. Figure \ref{fig:runtime} shows
ECDFs of run times of OCT-H and S$^2$OCT. OCT-H solved
\SI{86}{\percent} of the instances within the time limit, while
S$^2$OCT does so for \SI{58}{\percent}.
As expected, OCT-H has significantly shorter run times.

\begin{center}
  \begin{table}
    \caption{Different quantile values for the number of continuous
      and binary variables}
    \label{tablevariables}
    \begin{tabular}{rrrrr}
      \toprule
      & \multicolumn{2}{ c }{Continuous} &  \multicolumn{2}{ c }{  Binary } \\
      \cmidrule(lr){2-3}\cmidrule(lr){4-5}
       & OCT-H & S$^2$OCT & OCT-H & S$^2$OCT\\
      \midrule
      $\min$ & $31$ & $43$ & $42$ & $151$  \\
      $\SI{25}{\percent}$ & $61$ & $195$ & $123$ & $846$ \\
      $\SI{50}{\percent}$ & $97$ & $366$ & $226$ & $1843$  \\
      $\SI{75}{\percent}$ & $319$ & $3028$ & $1451$ & $16\,469$  \\
      $\max$ & $14\,039$ & $121\,910$ & $54\,373$ & $695\,835$ \\
      \bottomrule
    \end{tabular}
  \end{table}
\end{center}


\begin{figure}
  \centering
  \includegraphics[width=0.6\textwidth]{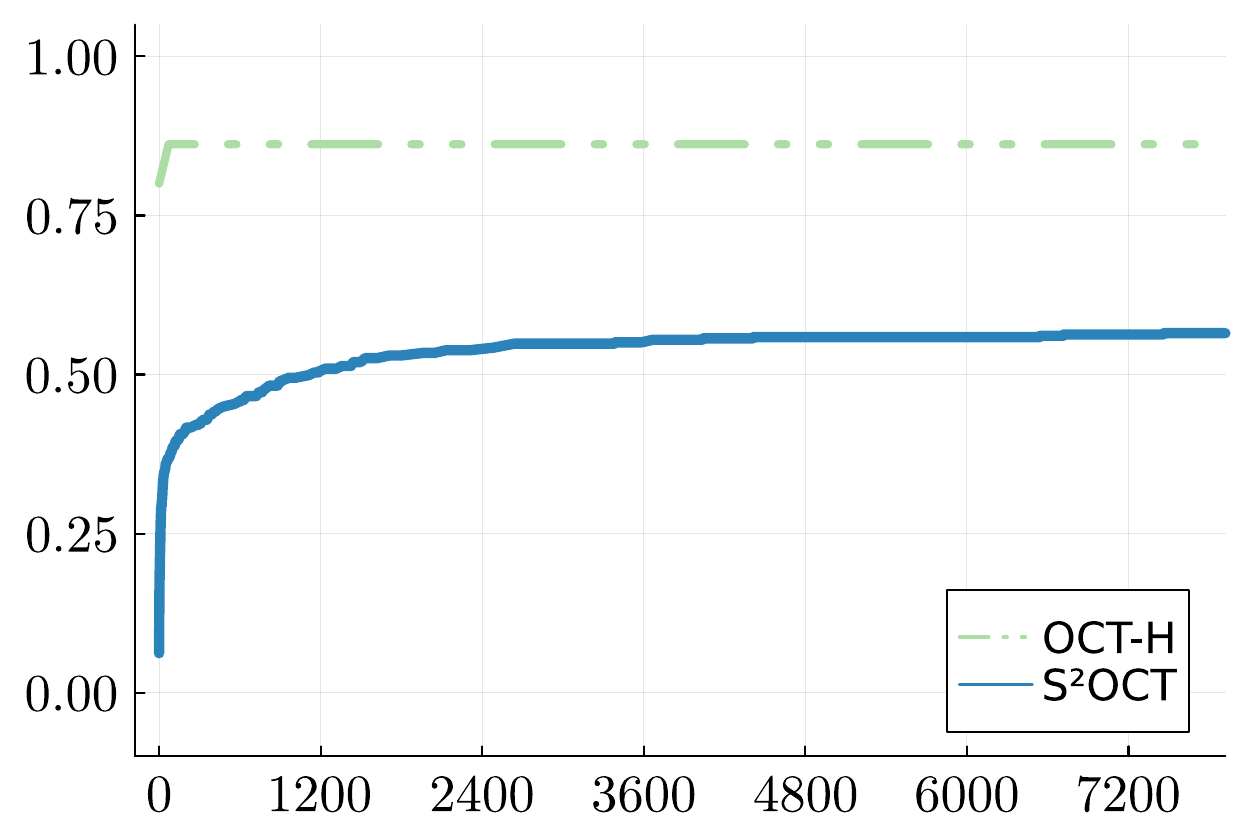}
  \caption{ECDFs for run times (in seconds).}
  \label{fig:runtime}
\end{figure}

\subsubsection{Accuracy and $\MCC$}
\label{sec:accuracyandMCC}
\begin{figure}[t]
  \centering
   \begin{tabular}{@{}>{\centering\arraybackslash}m{15pt}@{\hspace{1pt}}c@{\hspace{5pt}}c@{}}
        & \revTwo{All instances}  & \revTwo{Both terminate within TL} \\
        \rotatebox{90}{\revTwo{$\overline{\text{AC}}$ (All)}} &
        \includegraphics[width=0.45\textwidth, valign=m]{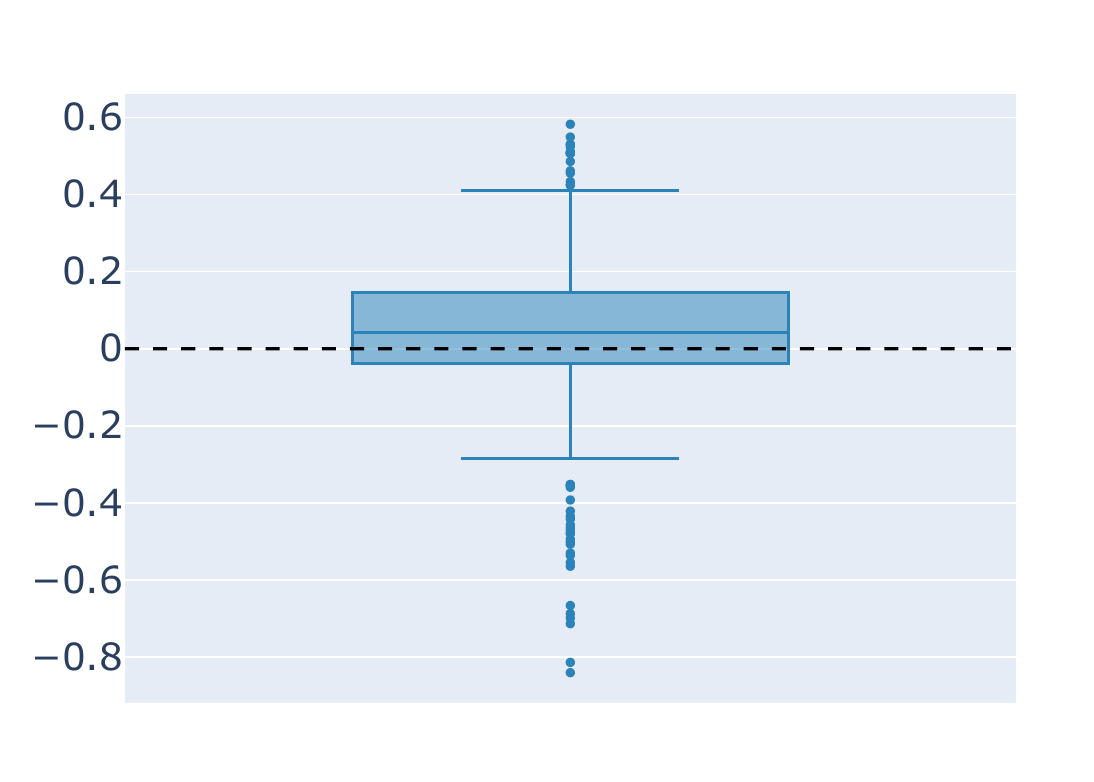} &
        \includegraphics[width=0.45\textwidth, valign=m]{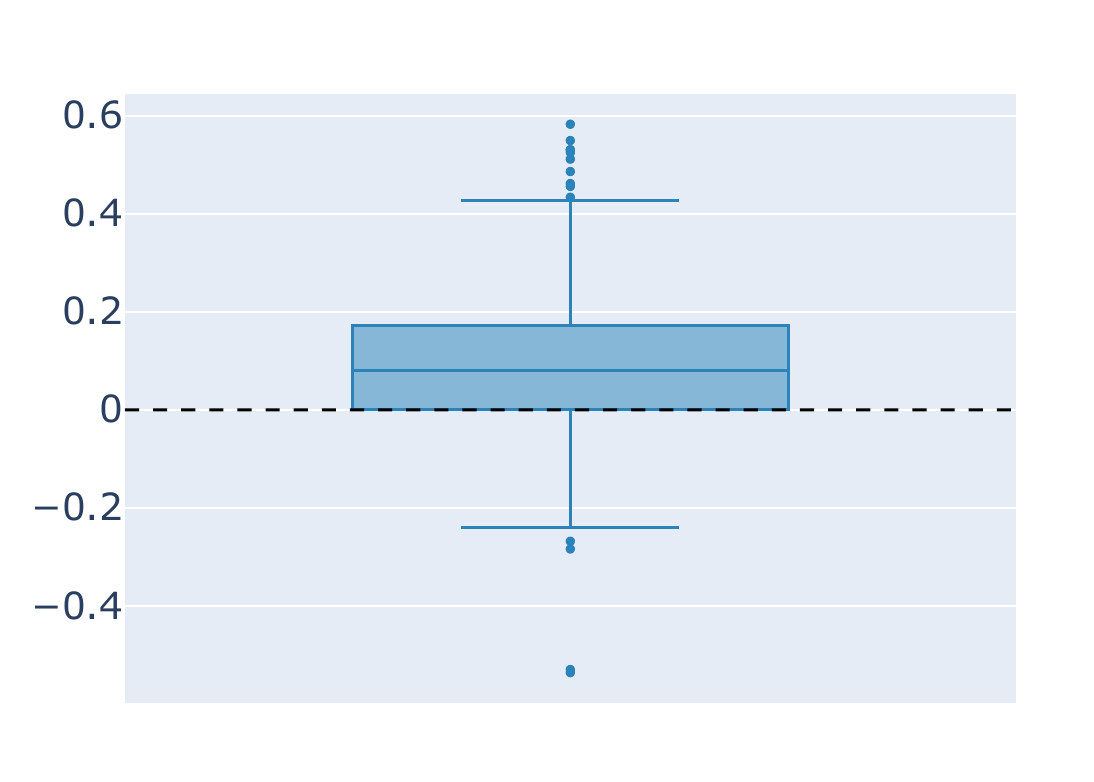} \\ [3ex]
        \rotatebox{90}{\revTwo{$\overline{\text{AC}}$ (Unlabeled)}} &
        \includegraphics[width=0.45\textwidth, valign=m]{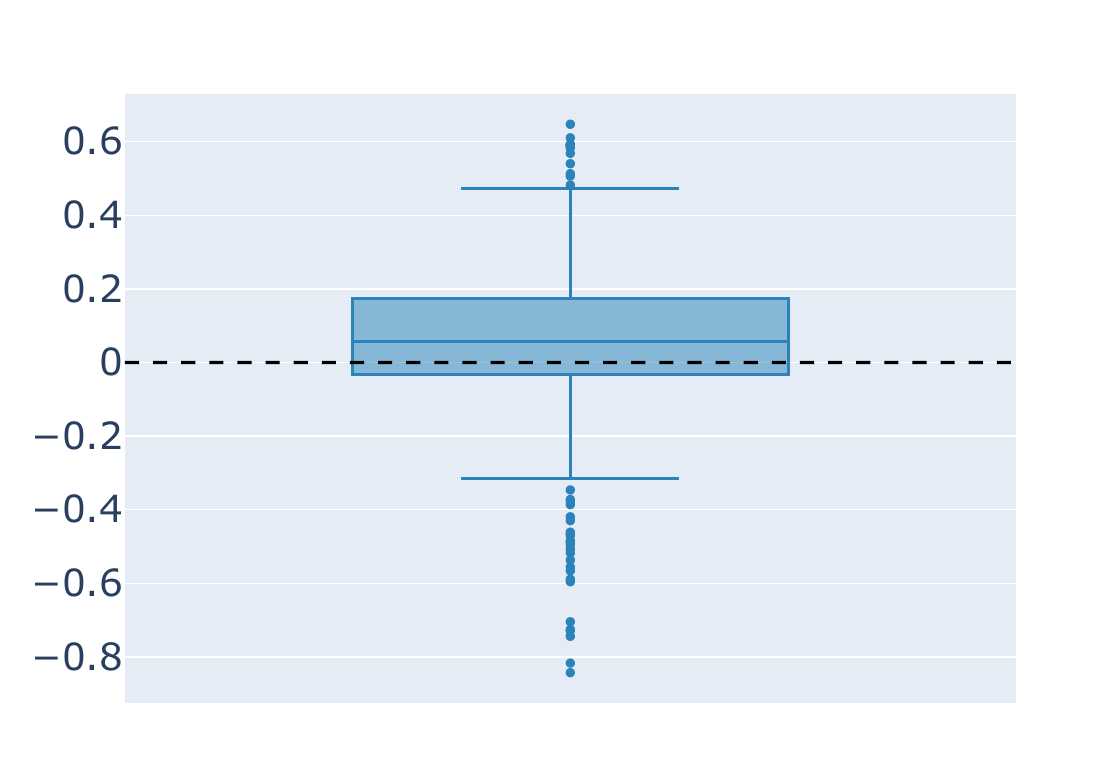} &
        \includegraphics[width=0.45\textwidth, valign=m]{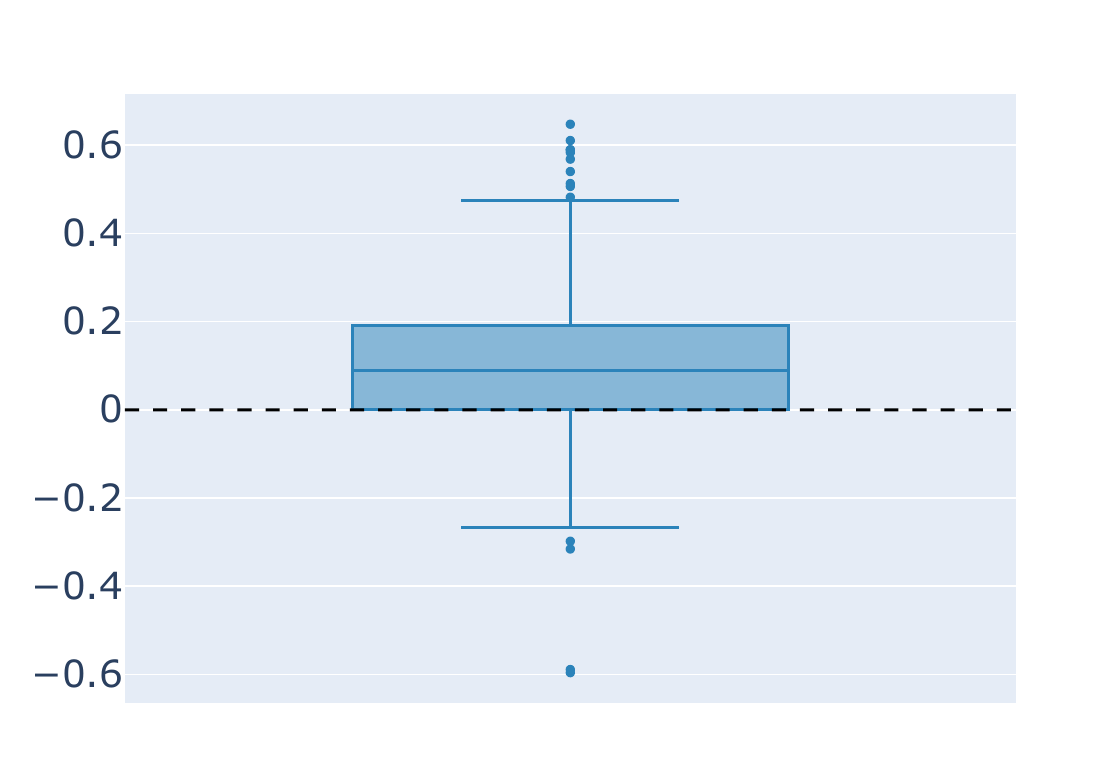} \\ [3ex]
        \rotatebox{90}{\revTwo{$\overline{\text{MCC}}$ (All)}} &
        \includegraphics[width=0.45\textwidth, valign=m]{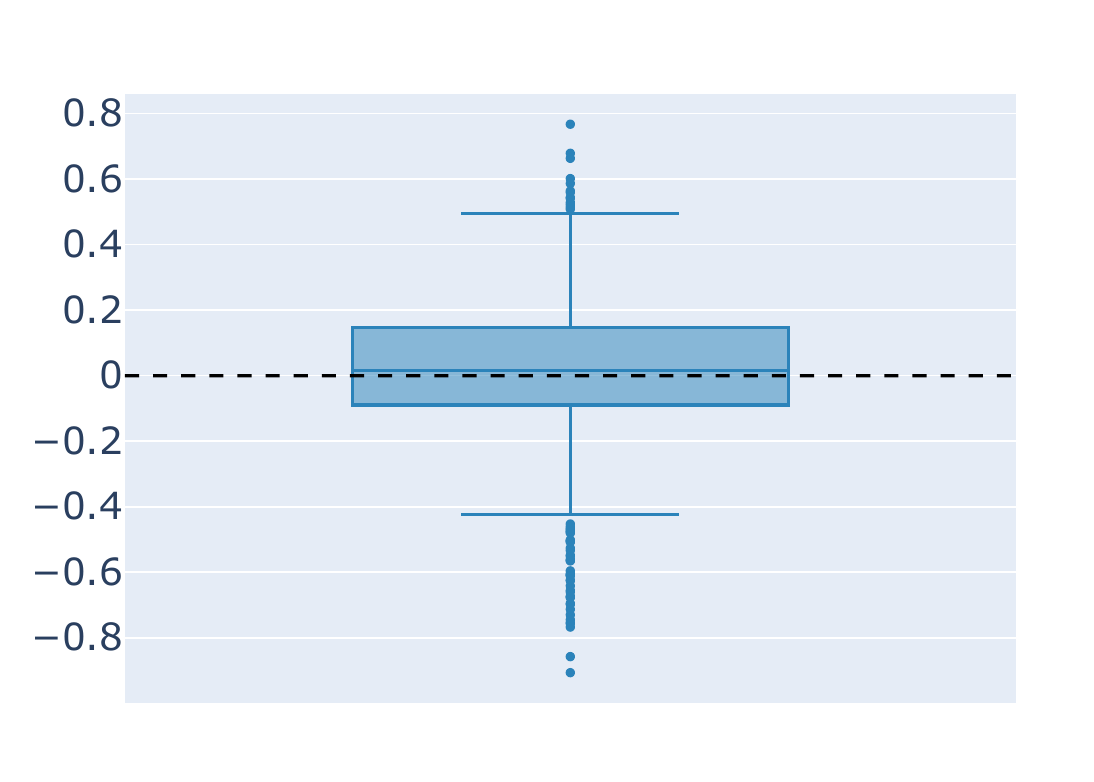} &
        \includegraphics[width=0.45\textwidth, valign=m]{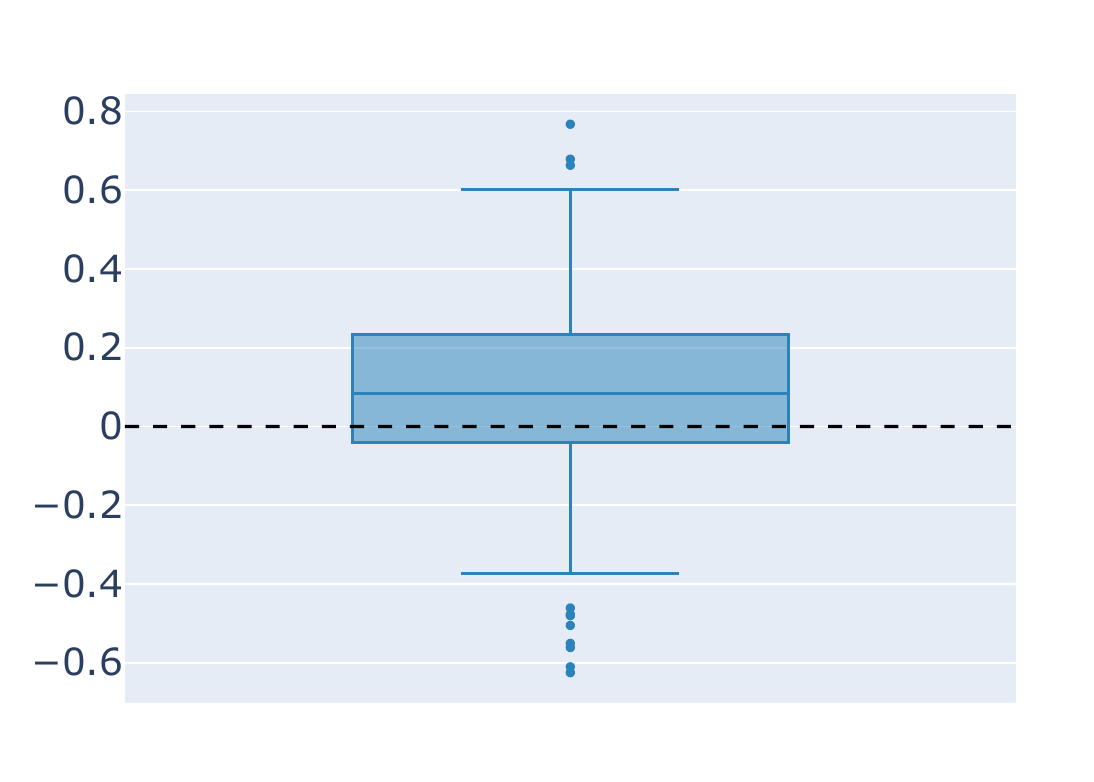} \\ [3ex]
        \rotatebox{90}{\revTwo{$\overline{\text{MCC}}$ (Unlabeled)}} &
        \includegraphics[width=0.45\textwidth, valign=m]{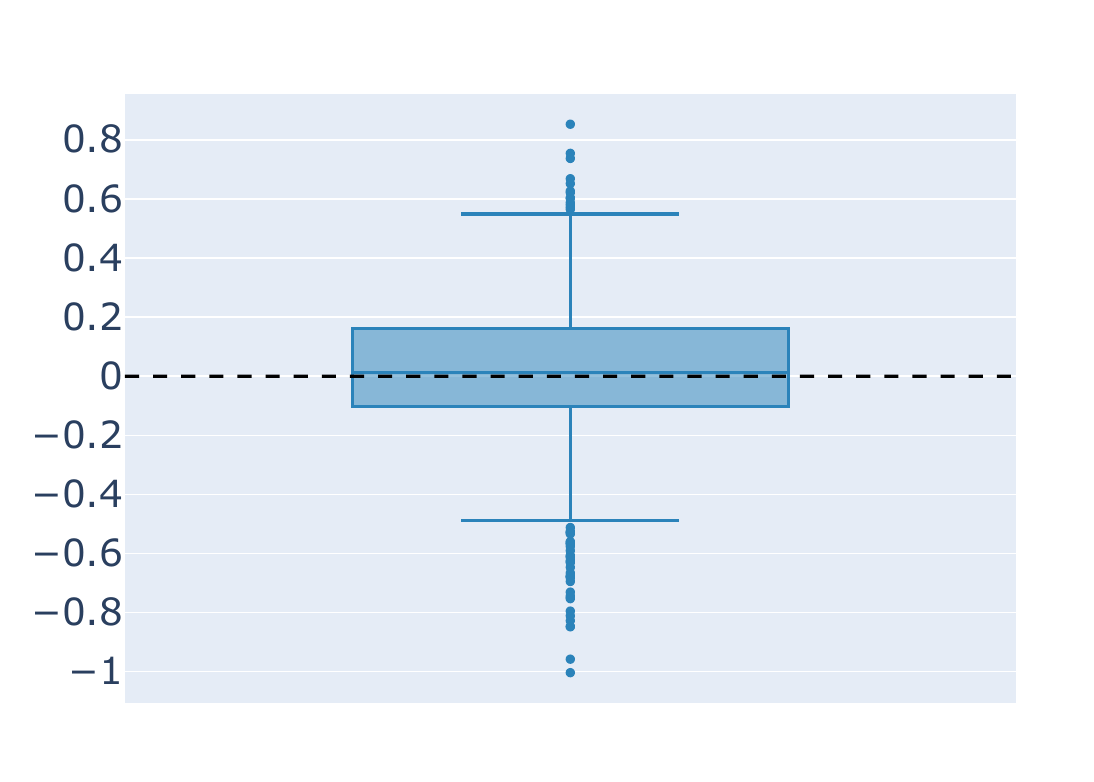} &
        \includegraphics[width=0.45\textwidth, valign=m]{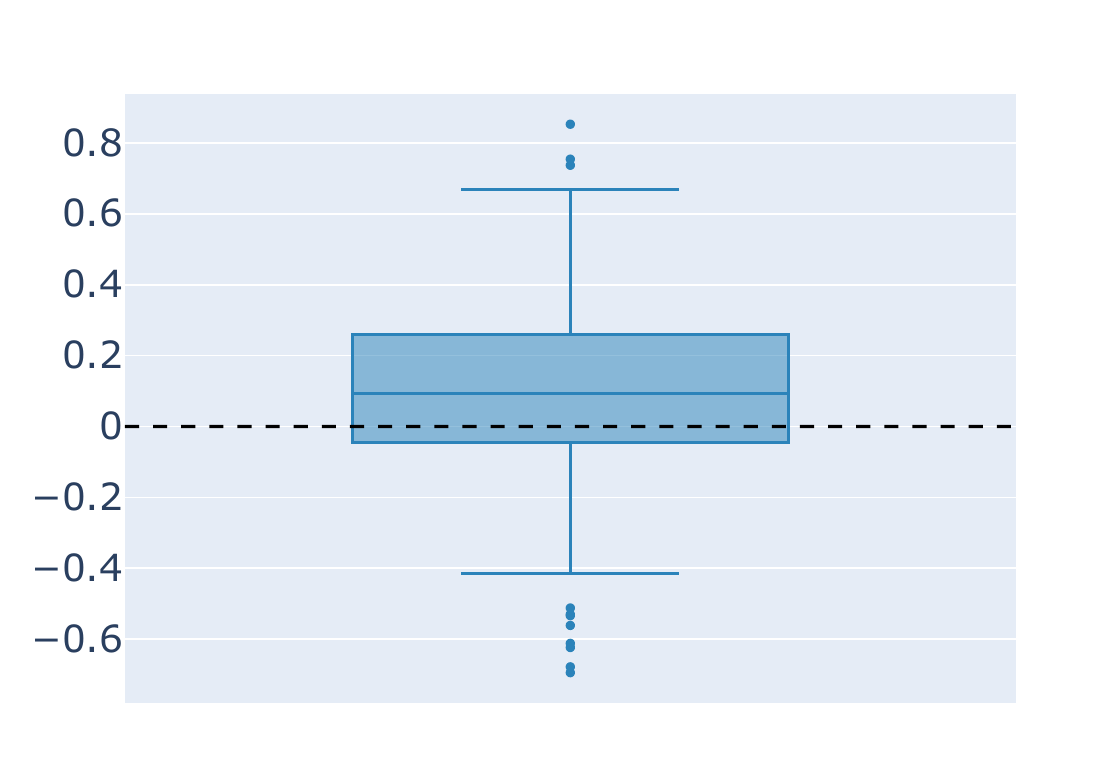}
    \end{tabular}
    \caption{\revTwo{Comparison of accuracy $\overline{\text{AC}}$ and
        Matthew's Correlation Coefficient $\overline{\text{MCC}}$ as
        described in \eqref{comparOCTH}. Rows distinguish between the
        entire dataset and unlabeled data, while columns contrast all
        instances versus those where both approaches finished within
        the time limit (TL).}}
    \label{fig:AC&MCC}
\end{figure}
Note that for both metrics, $\overline{\AC}$ and $\overline{\MCC}$, a
value greater than zero indicates that S$^2$OCT had a better result
than OCT-H and lower than zero\rev{, otherwise.}

\rev{Tables~\ref{tab:AC_results} and \ref{tab:MC_results} present the
  average and standard deviation of $\overline{\AC}$ and
  $\overline{\MCC}$, respectively.
Regarding the average values, note that S$^2$OCT performs better than
OCT-H in terms of both accuracy and MCC, especially for the instances
in which both approaches terminate within the time limit. With respect
to the standard deviation, we observe that the variability in both the
$\overline{\MCC}$ and $\overline{\AC}$ metrics decreases when
considering only the instances where both approaches terminated within
the time limit. This is expected, as once the method reaches the optimal
solution, the results tend to become more consistent in terms of
standard deviation across the successfully solved instances.}

\begin{table}[h!]
\centering
\caption{\rev{Average and Standard Deviation of $\overline{\AC}$; see
    Equation \eqref{comparOCTH}}}
\label{tab:AC_results}
\begin{tabular}{l c c c c}
\toprule
Instance & \multicolumn{2}{c}{All} & \multicolumn{2}{c}{Where both  approaches finish on time} \\
\cmidrule(lr){2-3} \cmidrule(lr){4-5}
 & Average & Std. Dev.   & Average  & Std. Dev.  \\
\midrule
Entire dataset &  0.041& 0.205 & 0.097  &0.165  \\
Unlabeled data & 0.056 & 0.222  &0.108 & 0.184 \\
\bottomrule
\end{tabular}
\end{table}
\begin{table}[h!]
\centering
\caption{\rev{Average and Standard Deviation of $\overline{\MCC}$\revTwo{; see}
    Equation \eqref{comparOCTH}}}
\label{tab:MC_results}
\begin{tabular}{l c c c c}
\toprule
Instance & \multicolumn{2}{c}{All} & \multicolumn{2}{c}{Where both  approaches finish on time} \\
\cmidrule(lr){2-3} \cmidrule(lr){4-5}
 & Average & Std. Dev.   & Average & Std. Dev. \\
\midrule
Entire dataset & 0.011 & 0.273   & 0.114   & 0.224 \\
Unlabeled data&  0.005  & 0.296  &0.106 & 0.253  \\
\bottomrule
\end{tabular}
\end{table}

\rev{We also provide the results in a boxplot.}
As can be seen in Figure~\ref{fig:AC&MCC}, the  $\overline{\AC}$
values  are greater than zero in  \SI{75}{\percent} of the results
(rows 1 and 2). Therefore, our proposed method takes advantage of
the additional information on the total number of cases for the
classes and has a better accuracy than OCT-H.
When comparing all instances (column 1), the negative outliers
indicate worse accuracy for S$^2$OCT than OCT-H in some cases. This
happens because in some instances our method does
not terminate within the time limit while OCT-H does. Since for those
instances that terminate within the time limit (column 2), we have few
outliers in accuracy (rows 1 and 2), we expect that the number of
instances with lower precision will decrease if we would increase the time
limit.
Figure~\ref{fig:AC&MCC} also shows that the $\overline{\MCC}$ values
are greater than zero in most cases (rows 3 and 4),  especially when
comparing only the instances that terminate in the time limit (column
2).   This means that our method has a better $\MCC$ than OCT-H. The
consequences of the results so far are that using the unlabeled points
as well as the cardinality
constraint allows to correctly classify the points with higher
accuracy and better $\MCC$ than with the optimal decision tree
approach OCT-H. Moreover, further numerical tests revealed that if
the percentage of labeled points is decreased, OCT-H tends to decrease
in accuracy and $\MCC$, while the deterioration for S$^2$OCT is much
less pronounced. This is especially relevant as in typical social
surveys the sample proportion is seldomly over \SI{1}{\percent} of the
population.


\section{Conclusion}
\label{sec:conclusion}

In many classification problems, acquiring labels for the entire
population of interest can be expensive.
Fortunately, external sources oftentimes can provide aggregated
information on how many points are in each class.
For this context, we proposed an MILP model for semi-supervised
multivariate OCTs that considers the setting of labeled and unlabeled
data points as well as additional aggregated information for the
unlabeled data for a binary classification.

Under the condition of simple random sampling, our proposed approach
has a slightly better accuracy and a better $\MCC$ than the
conventional optimal classification tree.
In many applications, however, the available data is coming from
non-probability samples, where the data collection mechanism is
largely unknown. Assuming simple random sampling in this setting is at
least optimistic.
Consequently, there is the risk of obtaining biased samples.
Our numerical results show that our model has better accuracy, $\MCC$,
and precision than the existing approach from the literature, even
with a small number of labeled points and biased samples. As expected,
the drawback of introducing the cardinality constraint is that we get
larger computational costs.
\rev{Consequently, one of the most important topics of future research
  is to develop problem-tailored solution strategies that lead to a
  better scalability of the presented MILP-based approach.
  Potential strategies might include, among others, decomposition
  strategies or primal heuristics to speed up the solution process.}

For \rev{further} future work, we will adapt our approach to a
multiclass OCT.
\rev{This could be achieved by adapting the model so that the leaf
  nodes consider more than two classes, treating each (point,
  class)-pair as a binary variable, which would inevitably increase the
  computational burden.}
Furthermore, more research is needed to further reduce the
computational burden.


\section*{Acknowledgements}

The authors thank the DFG for their support within RTG~2126
\enquote{Algorithmic Optimization}.


\printbibliography

\appendix
\section{Further Numerical Results}
\label{sec:furth-numer-results}

Besides the measures of accuracy and $\MCC$, we compare two further
measures that, depending on the application, can be highly relevant.
The first metric is precision ($\PR$).
It measures the proportion of correctly classified points among all
positively classified points and is thus defined as
\begin{equation}
  \label{prec}
  \PR \define \frac{\TP}{\TP + \FP} \in [0,1].
\end{equation}
This quantity is important in some application such as for fraud detection
systems, where identifying legitimate transactions as fraudulent is
better than identify fraudulent transactions as legitimate. Moreover,
precision can be higher when there are more positives in the dataset.

Second, we consider recall ($\RE$), which quantifies the proportion of
positive instances that are correctly classified as positive.
It is formally given by
\begin{equation}
  \label{recall}
  \RE \define \frac{\TP}{\TP+\FN} \in [0,1].
\end{equation}
This quantity is important in some applications such as in cancer
diagnosis, where evaluating recall is relevant as it is more
significant to identify potential cancer cases than to do
not. Different from precision, recall can be higher when there are
more negatives in the dataset.

As accuracy and $\MCC$, the main question is how much better
precision and recall of S$^2$OCT are compared to the one of the
OCT-H.
Hence, we compute the difference of the precision and recall according
to
\begin{equation}
  \label{comparOCTH2}
  \overline{\PR} \define \PR_{\SOCT}-\PR_{\OCTH}
  \quad
  \overline{\RE} \define \RE_{\SOCT}-\RE_{\OCTH},
\end{equation}
where $\PR_{\OCTH}$ and $\PR_{\SOCT}$ are computed as in~\eqref{prec}
for the OCT-H and S$^2$OCT, respectively. In the same way,
$\RE_{\OCTH}$ and $\RE_{\SOCT}$ are computed as in~\eqref{recall} for
the OCT-H and S$^2$OCT.
Note that as in Section \ref{sec:numerical-results}, for both
$\overline{\PR}$ and $\overline{\RE}$, a value greater than zero
indicates that S$^2$OCT has a better result than OCT-H and lower than
zero indicates that   S$^2$OCT has a worse result than OCT-H.
As can be seen in Figure~\ref{fig:PR&RE}, the $\overline{\PR}$
values  are greater than zero in more than  \SI{75}{\percent} of the
results (rows 1 and 2).  This means that  S$^2$OCT classifies the
points  with higher precision than OCT-H. Hence, OCT-H has more
false-positive results. The negative outliers most likely are due to
the same reason as  those for the respective $\overline{\AC}$ and
$\overline{\MCC}$ values.

On the other hand, Figure~\ref{fig:PR&RE} also show that
$\overline{\RE}$ is, in general, lower than $0$. This means that OCT-H
has better recall than our method. The results of precision and recall can
be justified by the fact that the biased sample is more likely to have
labeled data in class $\mathcal{A}$ and having no information about
the unlabeled data, the OCT-H ends up classifying points on the
positive side.

\begin{figure}[t]
    \centering
   \begin{tabular}{@{}>{\centering\arraybackslash}m{15pt}@{\hspace{1pt}}c@{\hspace{5pt}}c@{}}
        & \revTwo{All instances}  & \revTwo{Both terminate within TL} \\
        \rotatebox{90}{\revTwo{$\overline{\text{PR}}$ (All)}} &
        \includegraphics[width=0.45\textwidth, valign=m]{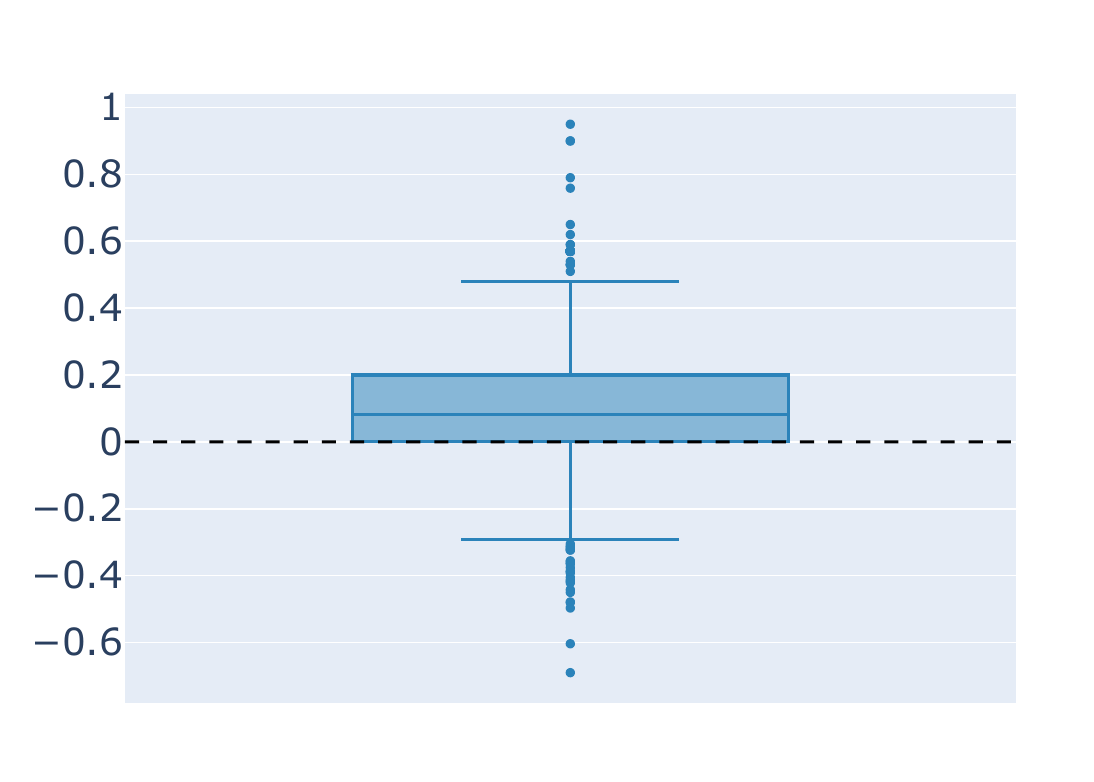} &
        \includegraphics[width=0.45\textwidth, valign=m]{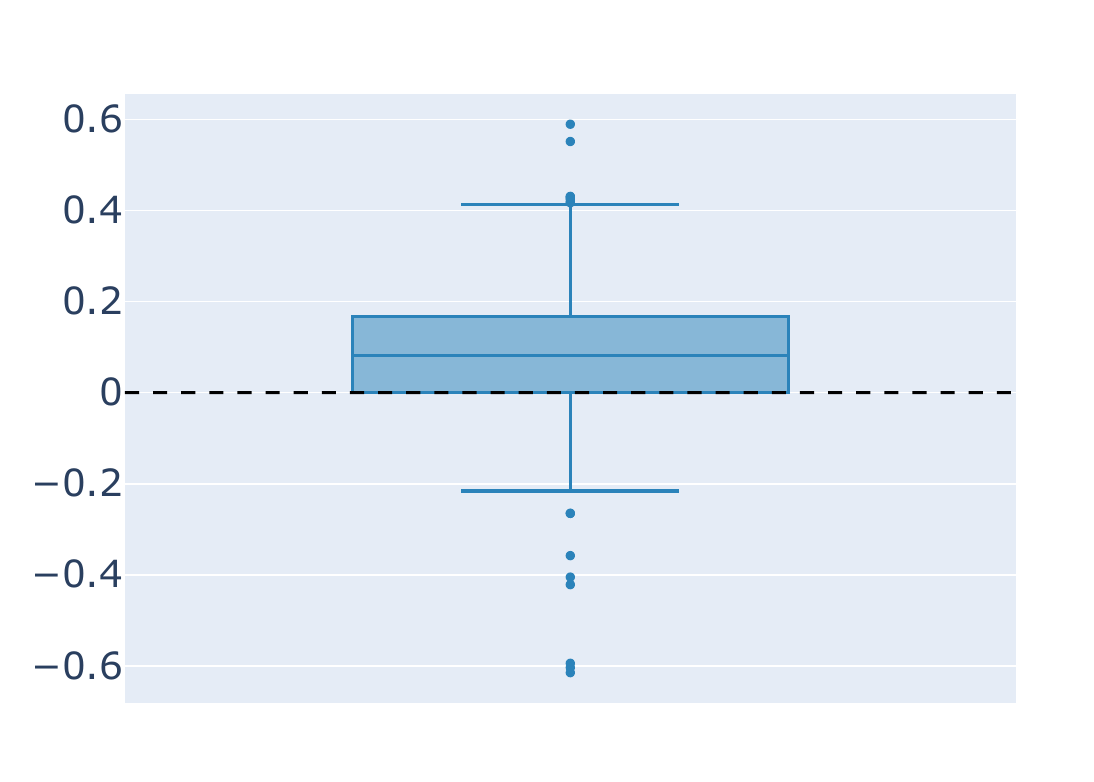} \\ [3ex]

        \rotatebox{90}{\revTwo{$\overline{\text{PR}}$ (Unlabeled)}} &
        \includegraphics[width=0.45\textwidth, valign=m]{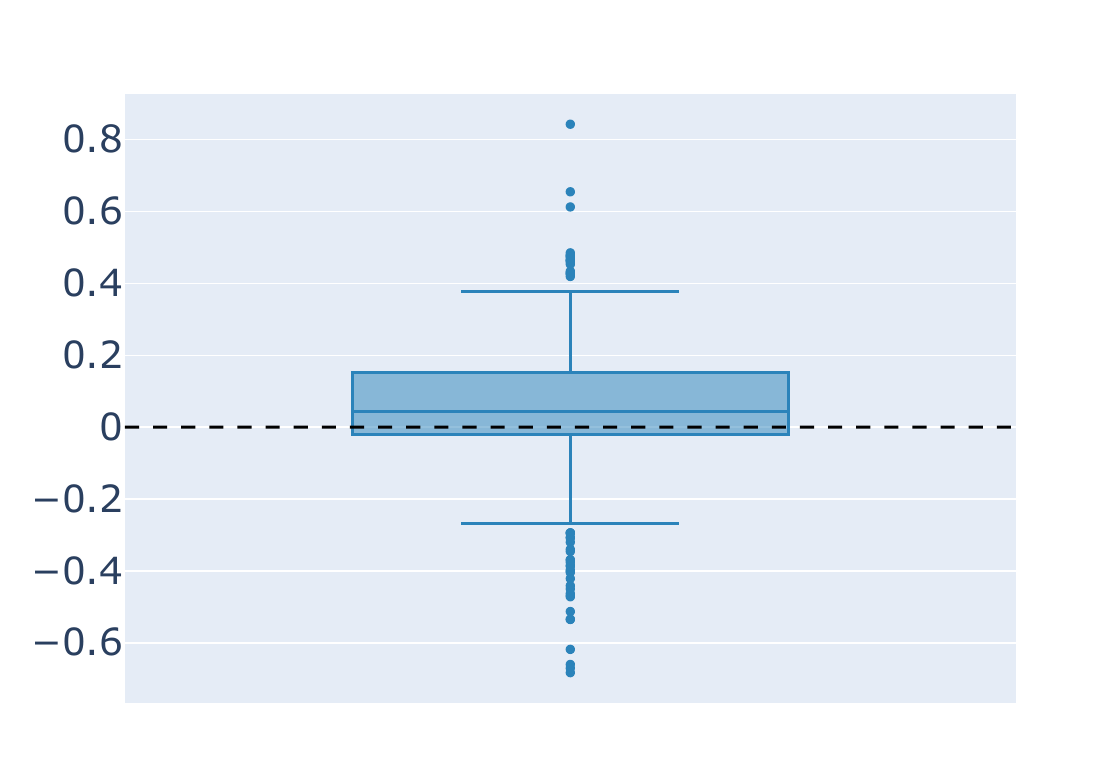} &
        \includegraphics[width=0.45\textwidth, valign=m]{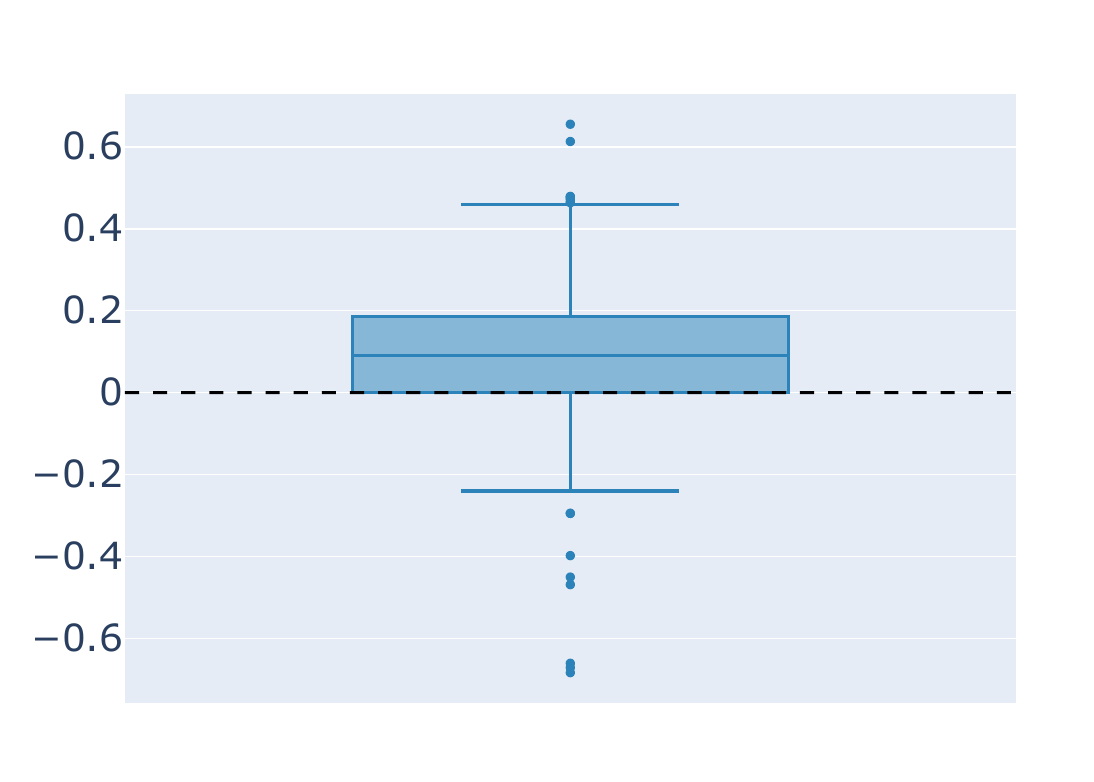} \\ [3ex]

        \rotatebox{90}{\revTwo{$\overline{\text{RE}}$ (All)}} &
        \includegraphics[width=0.45\textwidth, valign=m]{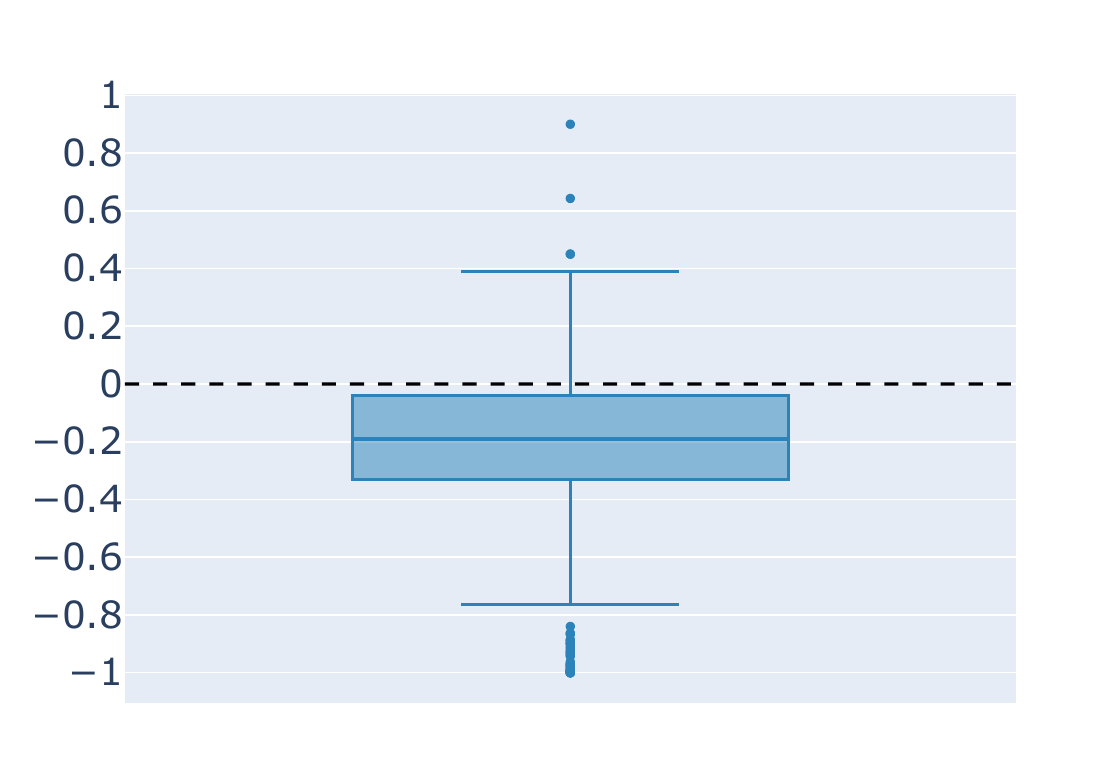} &
        \includegraphics[width=0.45\textwidth, valign=m]{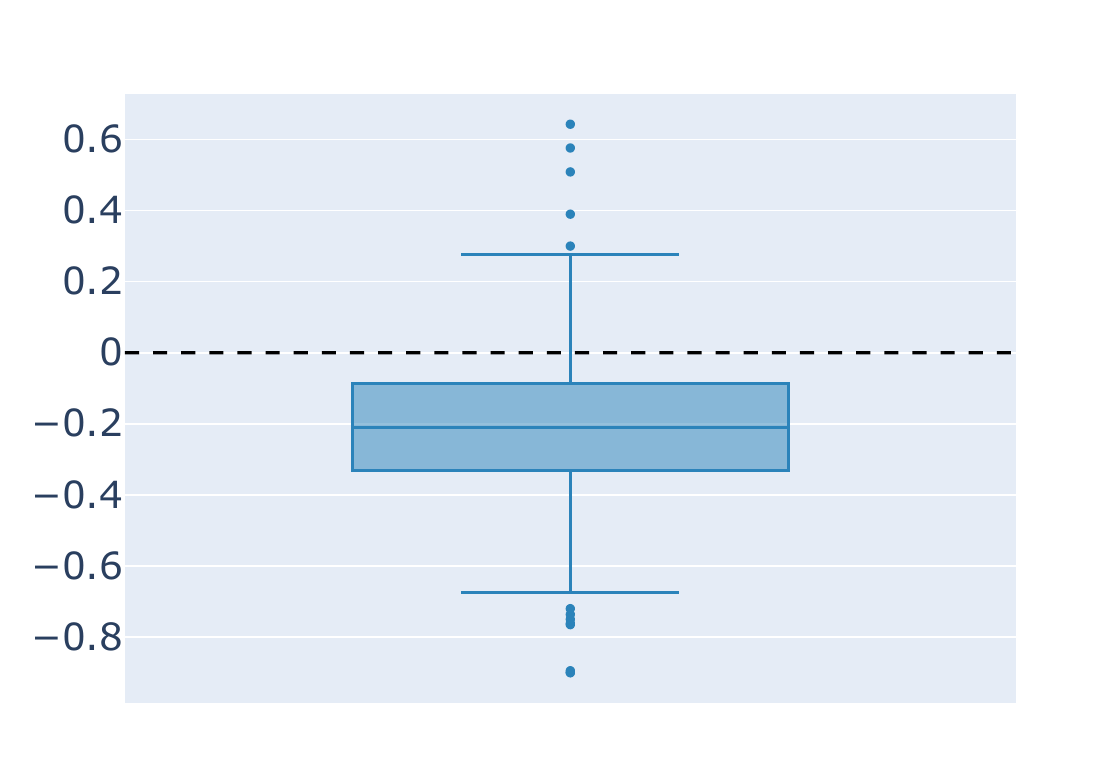} \\ [3ex]

        \rotatebox{90}{\revTwo{$\overline{\text{RE}}$ (Unlabeled)}} &
        \includegraphics[width=0.45\textwidth, valign=m]{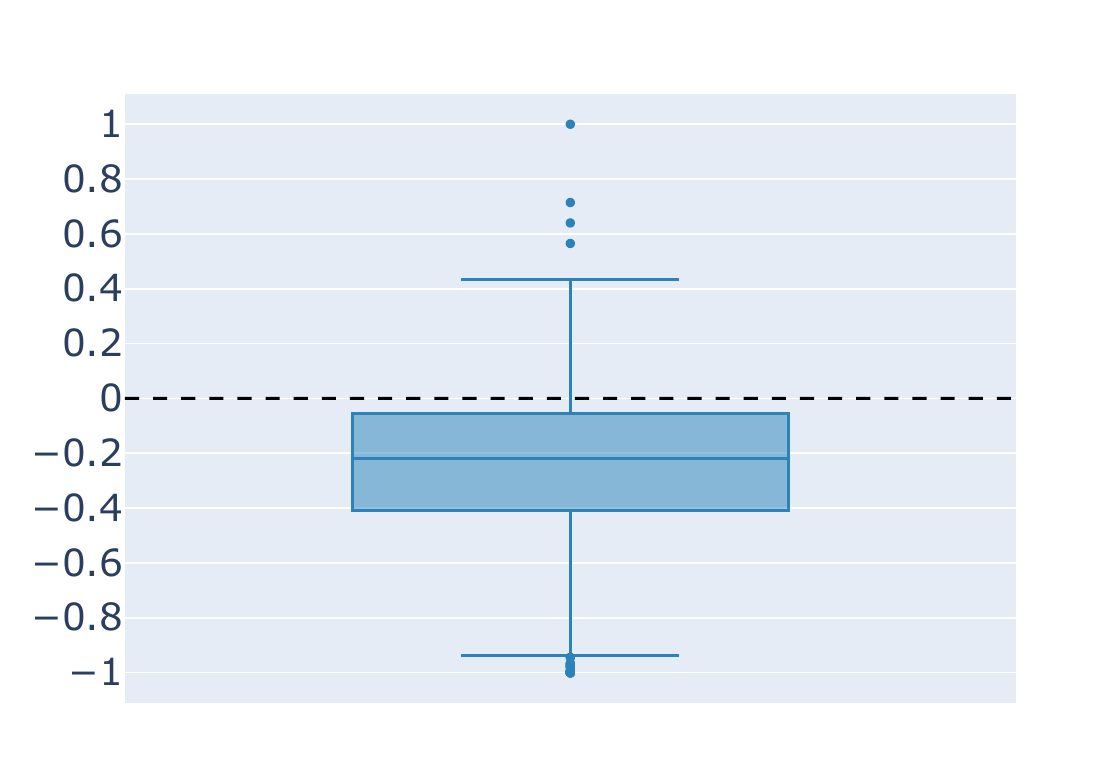} &
        \includegraphics[width=0.45\textwidth, valign=m]{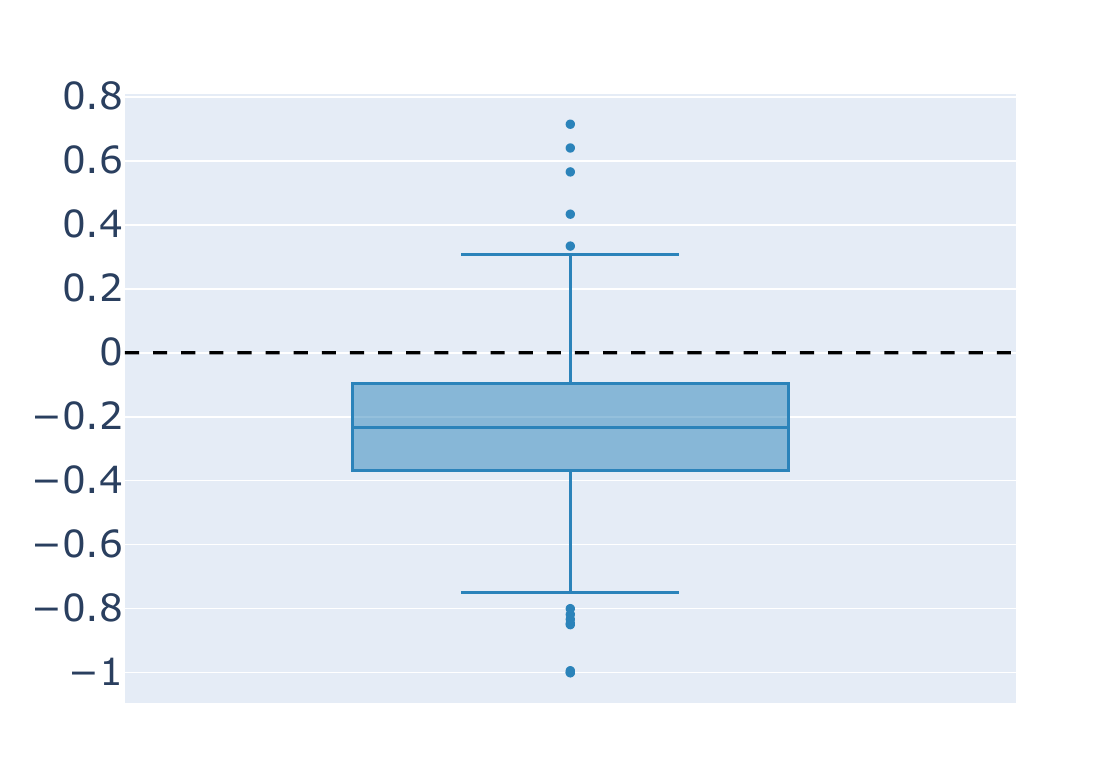}
    \end{tabular}

    \caption{\revTwo{Comparison of precision $\overline{\text{PR}}$ and recall $\overline{\text{RE}}$ as described in \eqref{comparOCTH2}. Rows distinguish between the entire data set and unlabeled data, while columns contrast all instances versus those where both approaches finished within the time limit (TL).}}
    \label{fig:PR&RE}
\end{figure}


\section{Numerical Results for Simple random samples}
\label{sec:num-results-simple-sample}

Our computational study in Section~\ref{sec:numerical-results}
focuses on the analysis of non-representative biased samples.
The typical baseline scenario for evaluating the performance of
estimators is to apply them on simple random samples. Therefore, to
complement our numerical results, we also present the results under
simple random sampling. In a simple random sampling, each
unit in the data set has the same probability $\pi_i = n/N$ to be
included in the sample of labeled data of size $n$.
The instances are the same as described in
Section~\ref{subsection-test-sets}.
The computational setup follows the description in
Section~\ref{subsection-comp-setup}. As before,  the used evaluation
criteria are $\overline{\AC}$ and $\overline{\MCC}$ as
in~\eqref{comparOCTH} and $\overline{\PR}$ and $\overline{\RE}$ as
in~\eqref{comparOCTH2}.

\begin{figure}[t]
    \centering
   \begin{tabular}{@{}>{\centering\arraybackslash}m{15pt}@{\hspace{1pt}}c@{\hspace{5pt}}c@{}}
        & \revTwo{All instances}  & \revTwo{Both terminate within TL} \\
        \rotatebox{90}{\revTwo{$\overline{\text{AC}}$ (All)}} &
        \includegraphics[width=0.45\textwidth, valign=m]{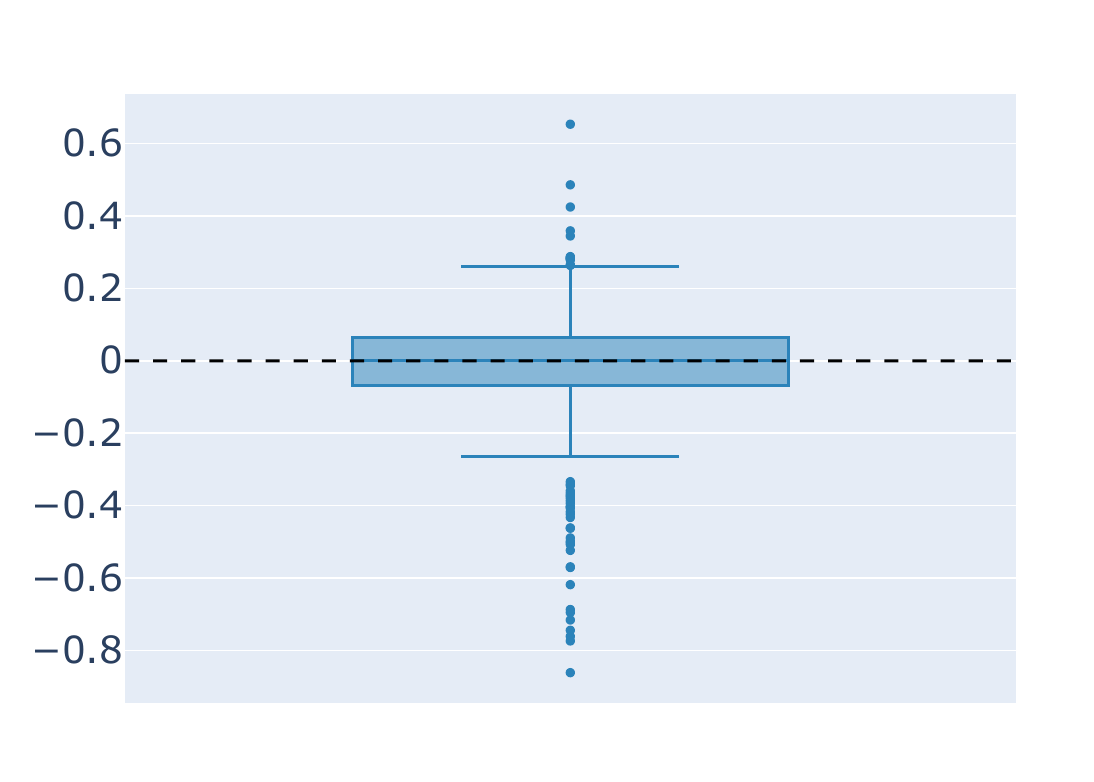} &
        \includegraphics[width=0.45\textwidth, valign=m]{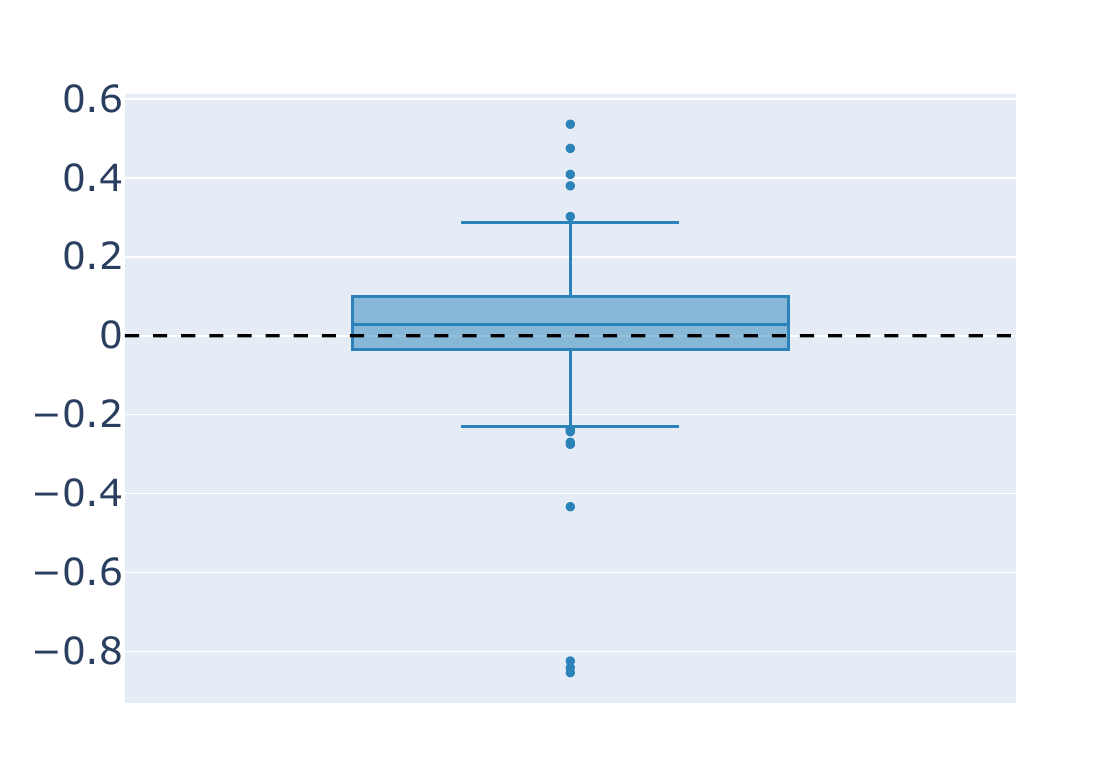} \\ [3ex]

        \rotatebox{90}{\revTwo{$\overline{\text{AC}}$ (Unlabeled)}} &
        \includegraphics[width=0.45\textwidth, valign=m]{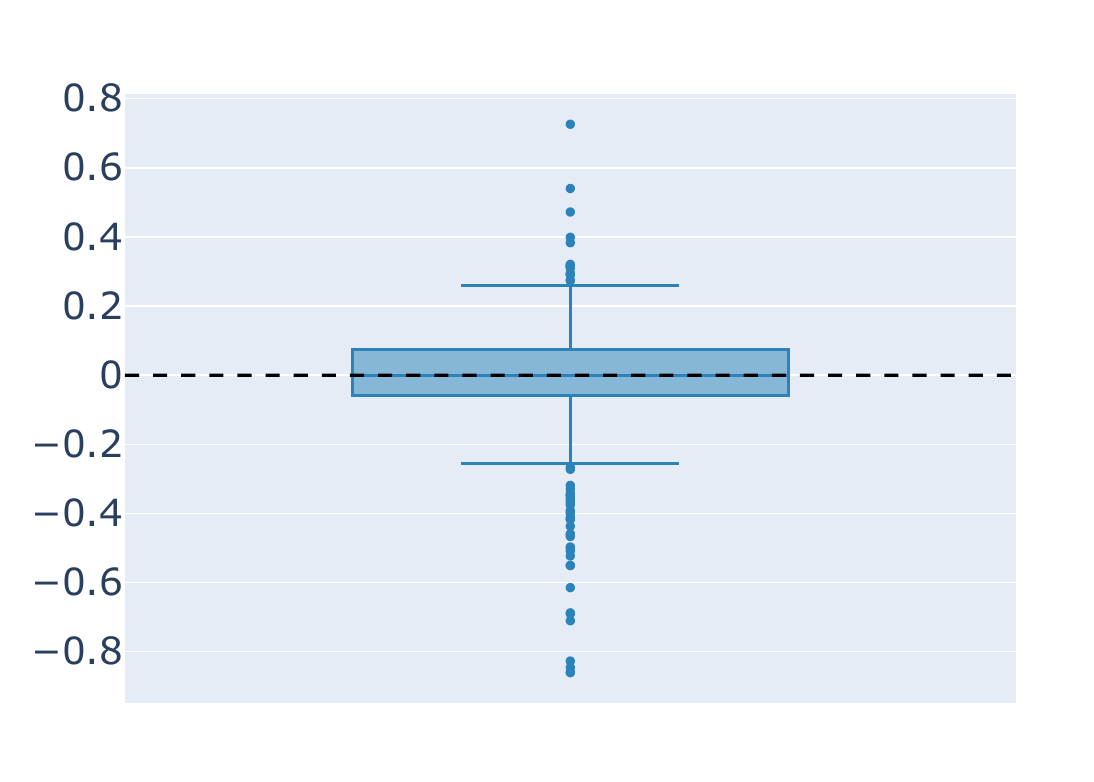} &
        \includegraphics[width=0.45\textwidth, valign=m]{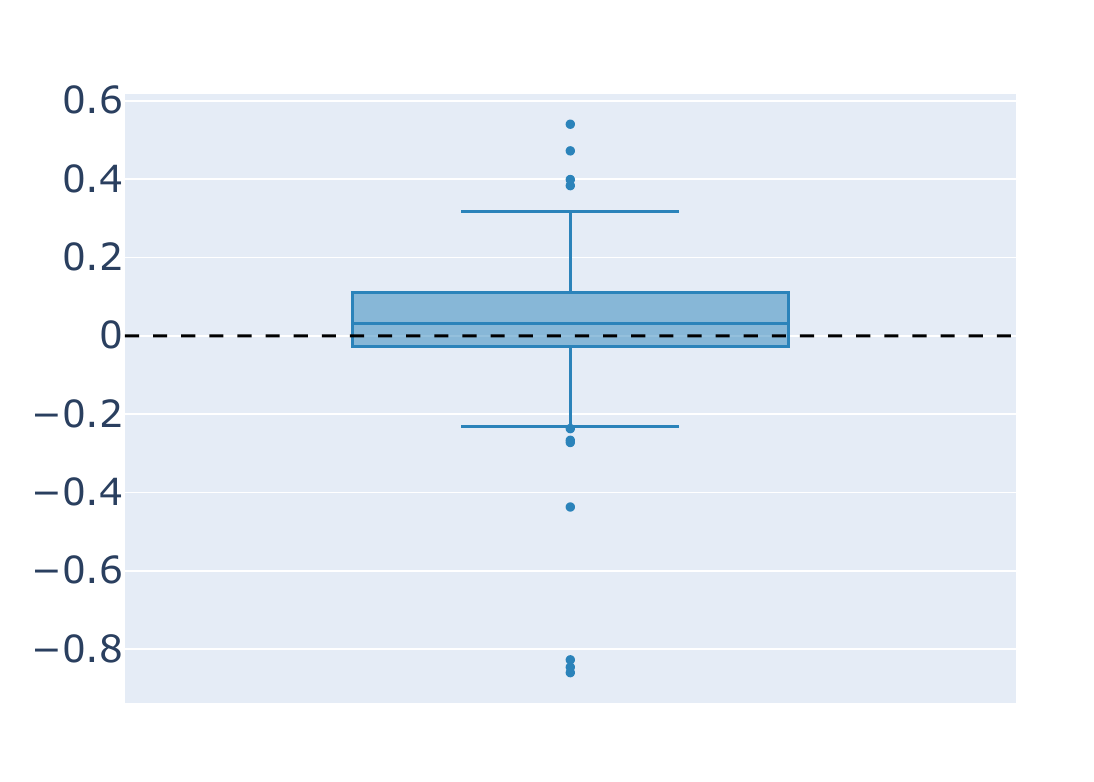} \\ [3ex]

        \rotatebox{90}{\revTwo{$\overline{\text{MCC}}$ (All)}} &
        \includegraphics[width=0.45\textwidth, valign=m]{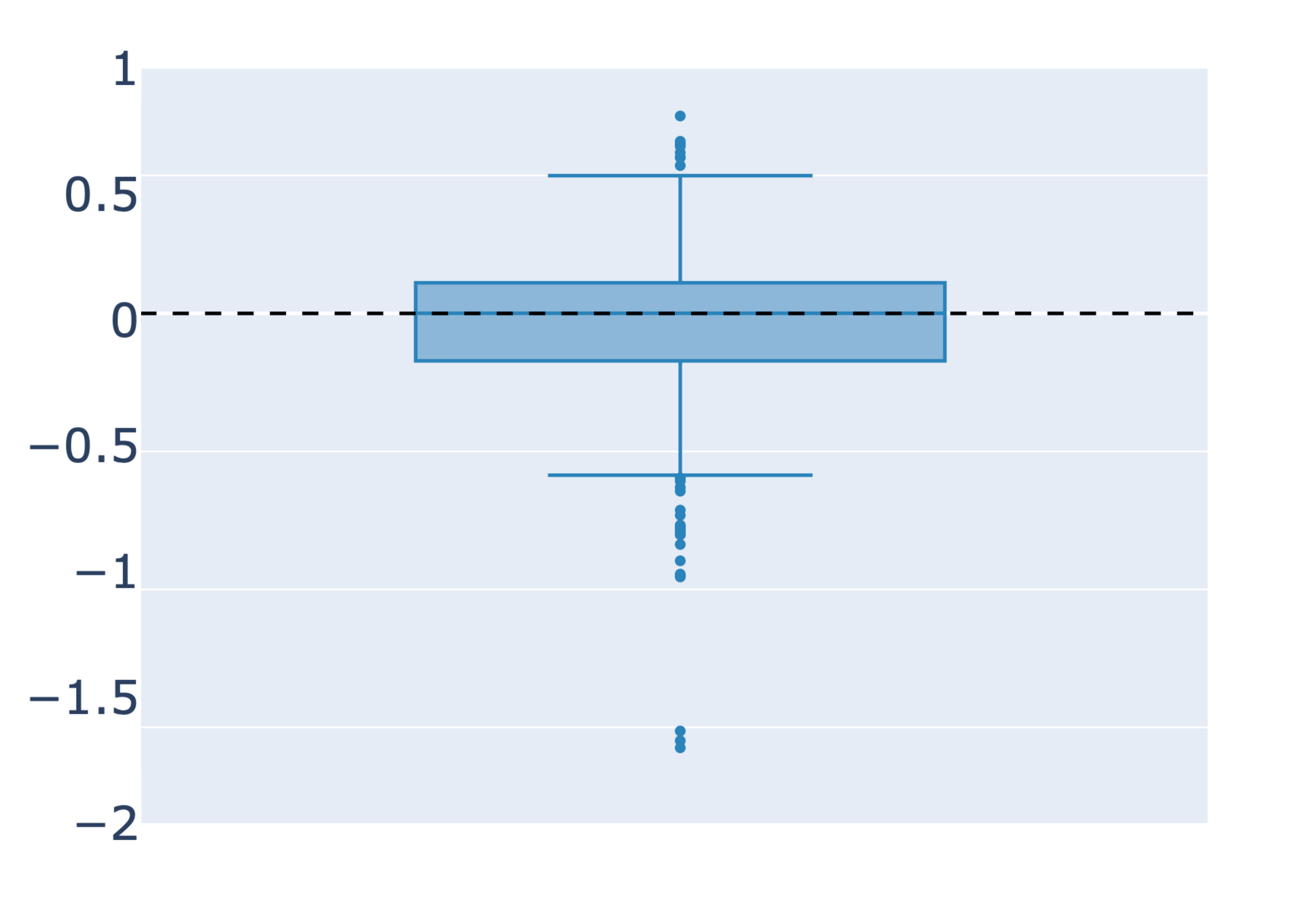} &
        \includegraphics[width=0.45\textwidth, valign=m]{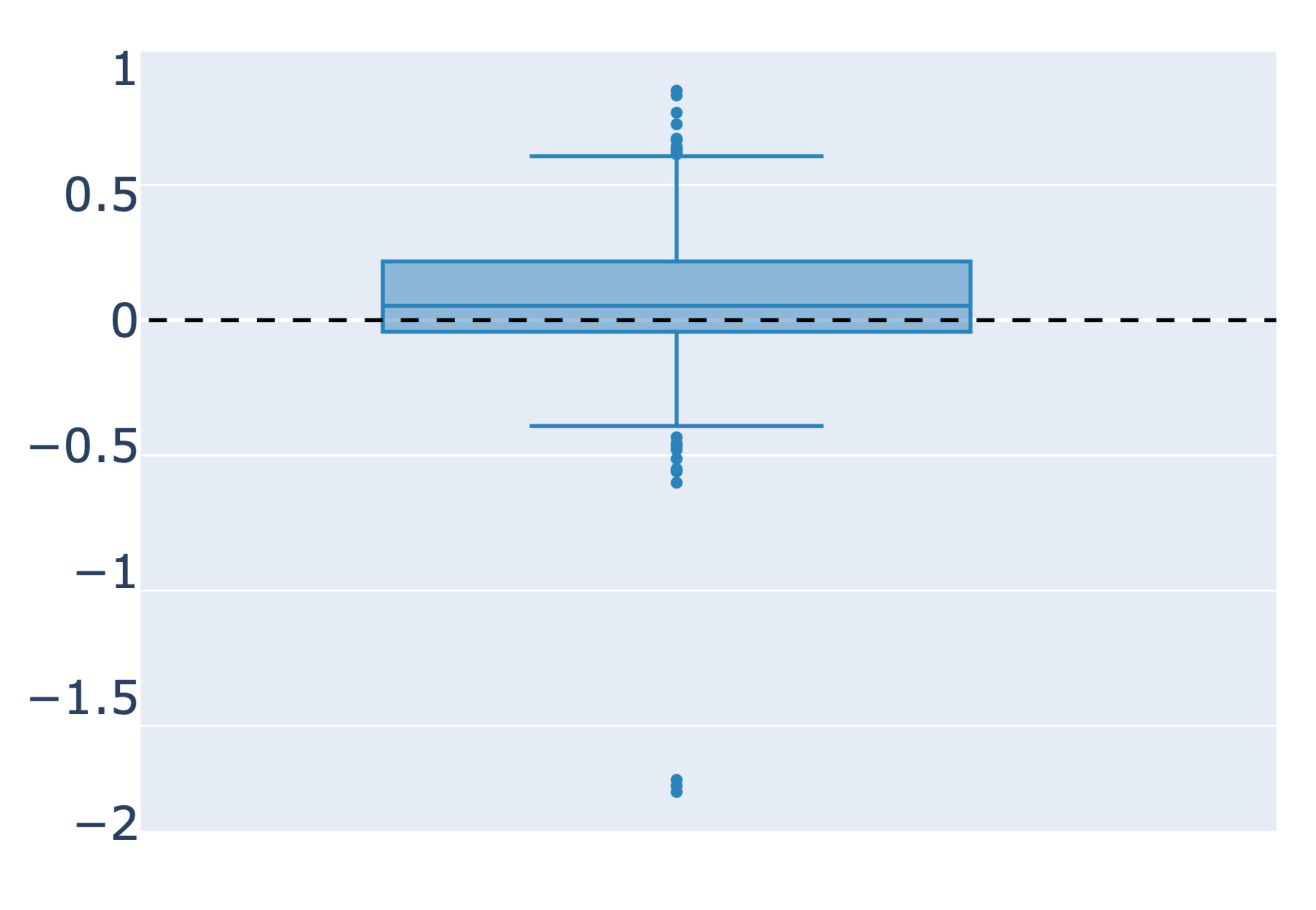} \\ [3ex]

        \rotatebox{90}{\revTwo{$\overline{\text{MCC}}$ (Unlabeled)}} &
        \includegraphics[width=0.45\textwidth, valign=m]{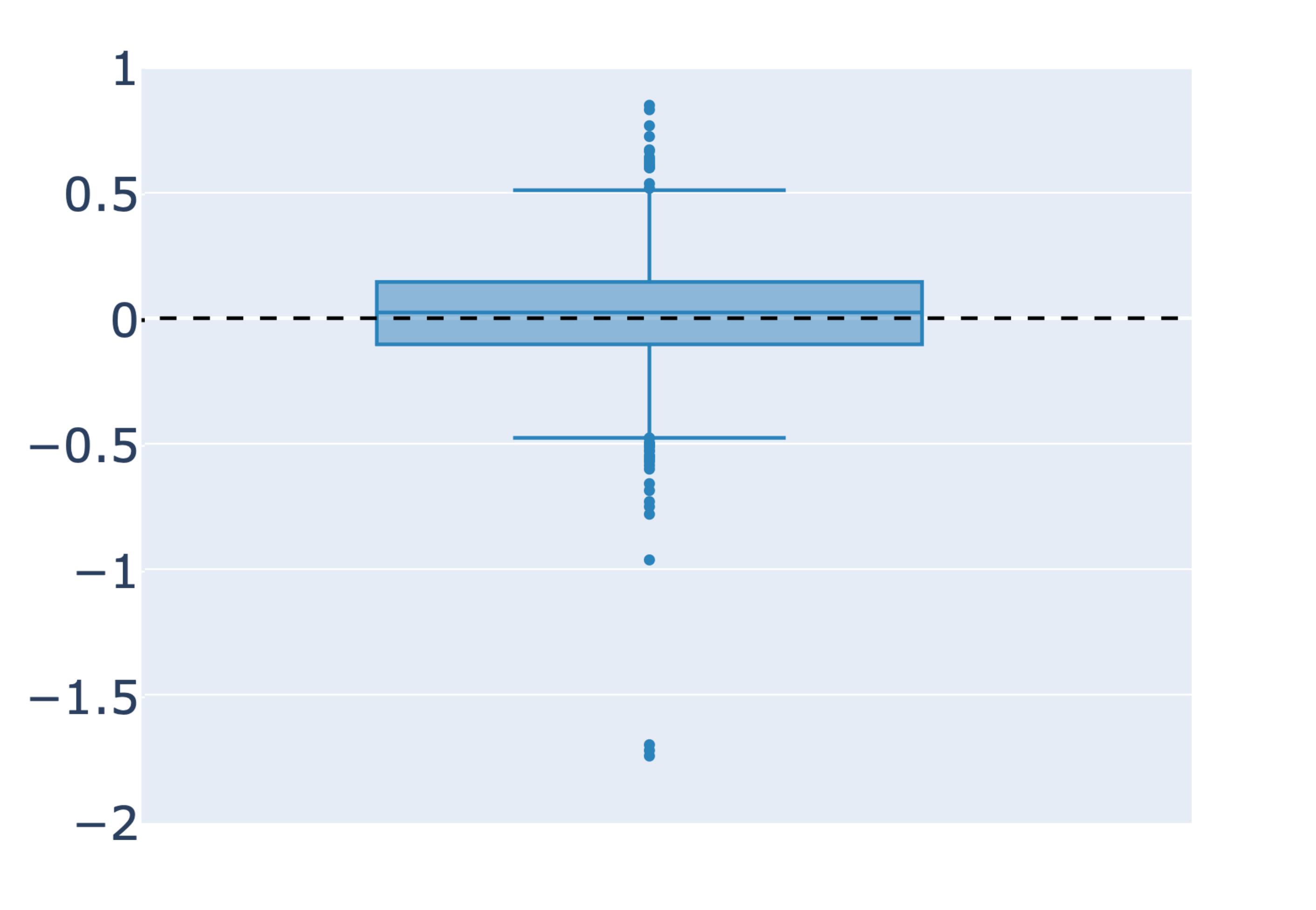} &
        \includegraphics[width=0.45\textwidth, valign=m]{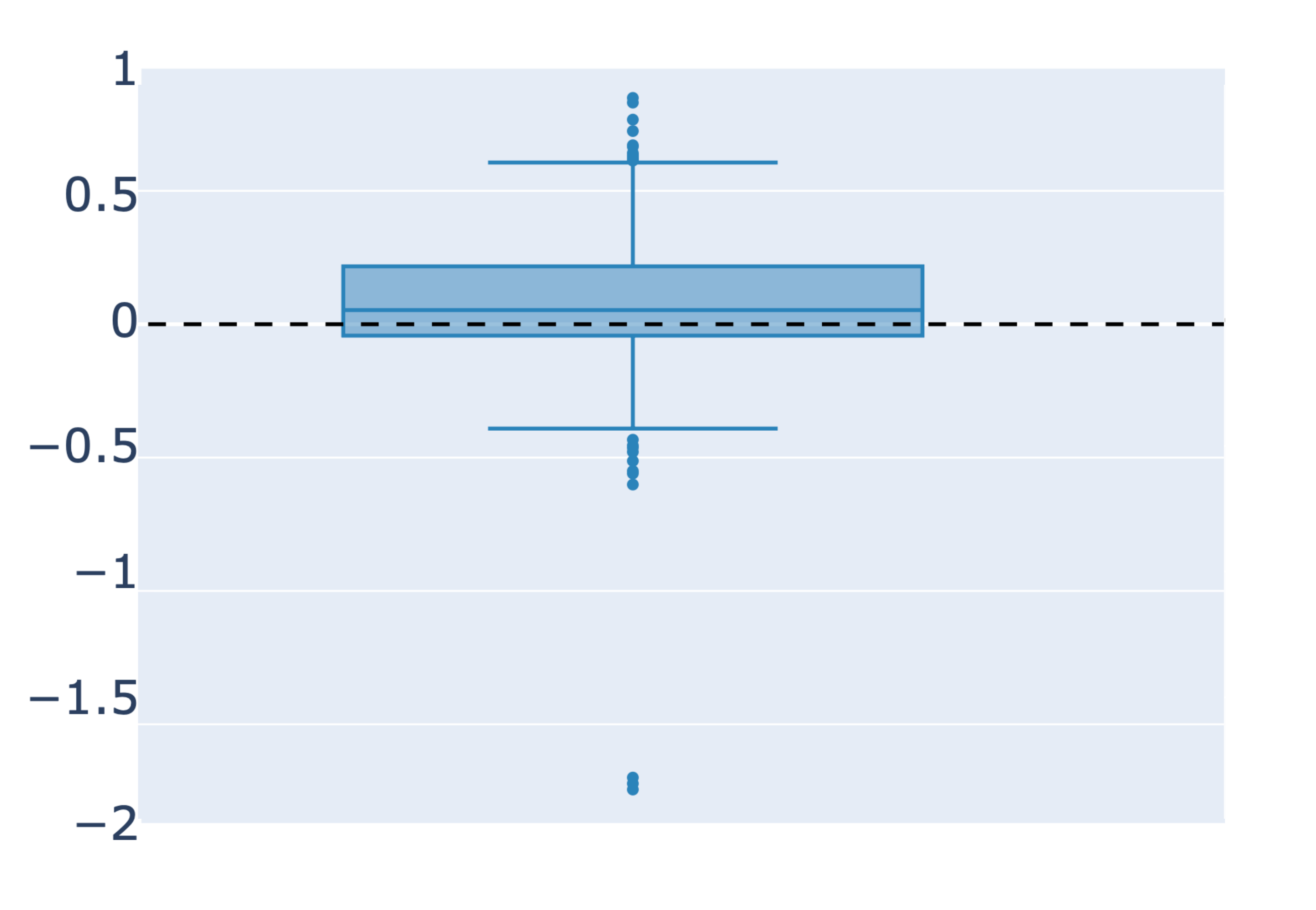}
    \end{tabular}

    \caption{\revTwo{Comparison of accuracy $\overline{\text{AC}}$ and Matthew's Correlation Coefficient $\overline{\text{MCC}}$ for the simple sample approach as described in \eqref{comparOCTH}. Rows distinguish between the entire data set and unlabeled data, while columns contrast all instances versus those where both approaches finished within the time limit (TL).}}
    \label{fig:AC&MCCsimplesample}
\end{figure}
Figure~\ref{fig:AC&MCCsimplesample} shows that for all the instances
(column 1), $\overline{\AC}$ (rows 1 and 2) and $\overline{\MCC}$
(rows 3 and 4) have value greater than 0 and lower than 0 in
\SI{50}{\percent} of the cases. This means both approaches have similar
accuracy and $\MCC$. However,  when comparing only those instances that
terminate within the time limit (column 2), it can be seen that S$^2$OCT
has slightly better accuracy and $\MCC$, but not as much as for biased
samples; see Section~\ref{sec:accuracyandMCC}. This is expected
because the sample is not biased.  Consequently, the cardinality
constraint, which aims to balance the class distribution, does not
introduce additional meaningful information to the problem. As can be
seen in Figure~\ref{fig:PR&RESimpleSample}, precision and recall are
similar for both approaches. Therefore, for the simple random samples,
our approach has almost the same results as OCT-H, with slight
improvements in accuracy and $\MCC$.
\begin{figure}[t]
    \centering
   \begin{tabular}{@{}>{\centering\arraybackslash}m{15pt}@{\hspace{1pt}}c@{\hspace{5pt}}c@{}}
        & \revTwo{All instances}  & \revTwo{Both terminate within TL} \\
        \rotatebox{90}{\revTwo{$\overline{\text{PR}}$ (All)}} &
        \includegraphics[width=0.45\textwidth, valign=m]{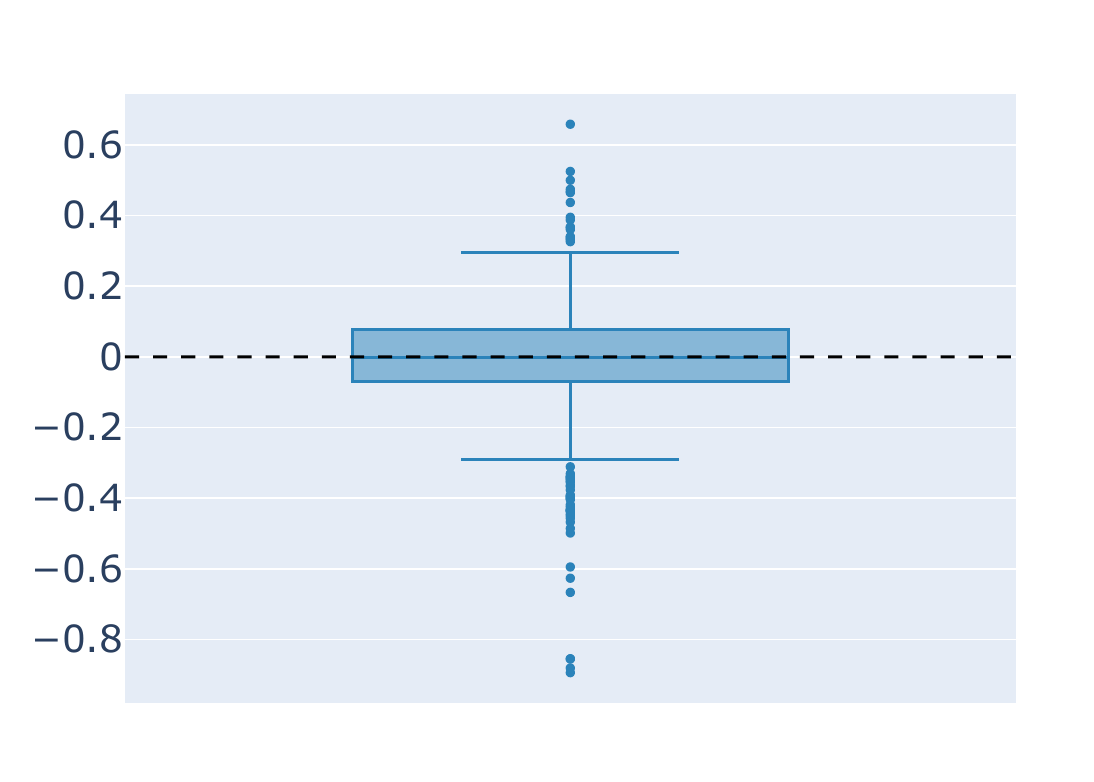} &
        \includegraphics[width=0.45\textwidth, valign=m]{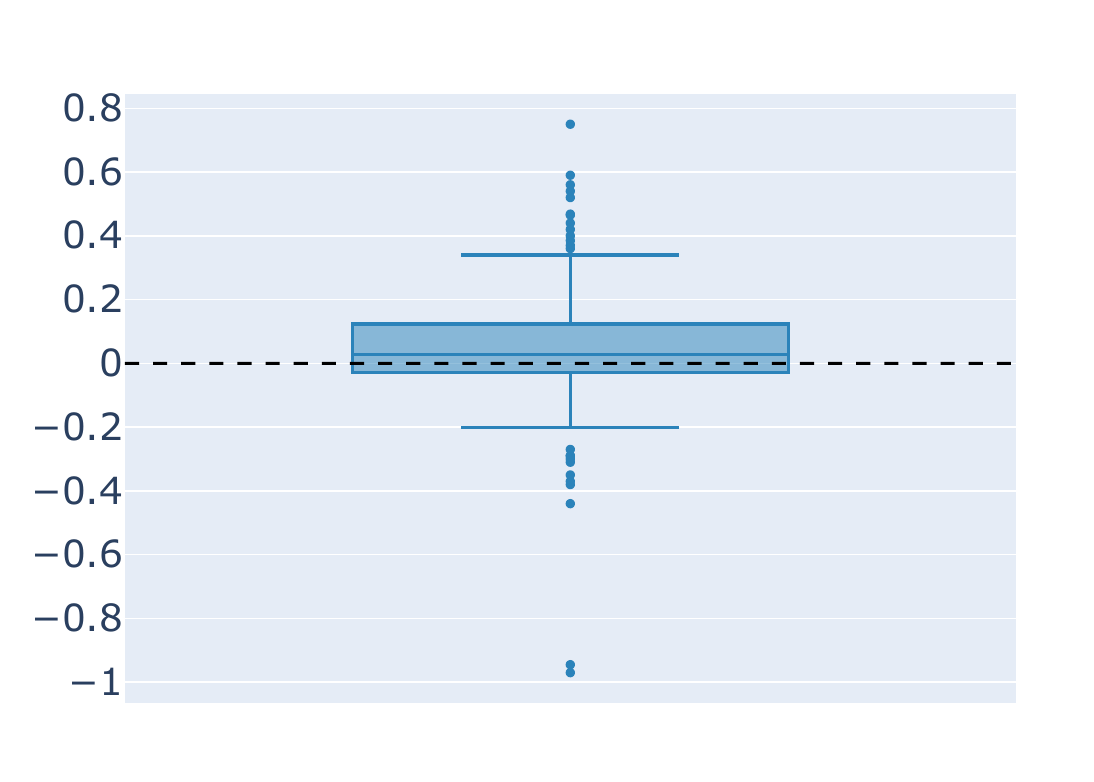} \\ [3ex]

        \rotatebox{90}{\revTwo{$\overline{\text{PR}}$ (Unlabeled)}} &
        \includegraphics[width=0.45\textwidth, valign=m]{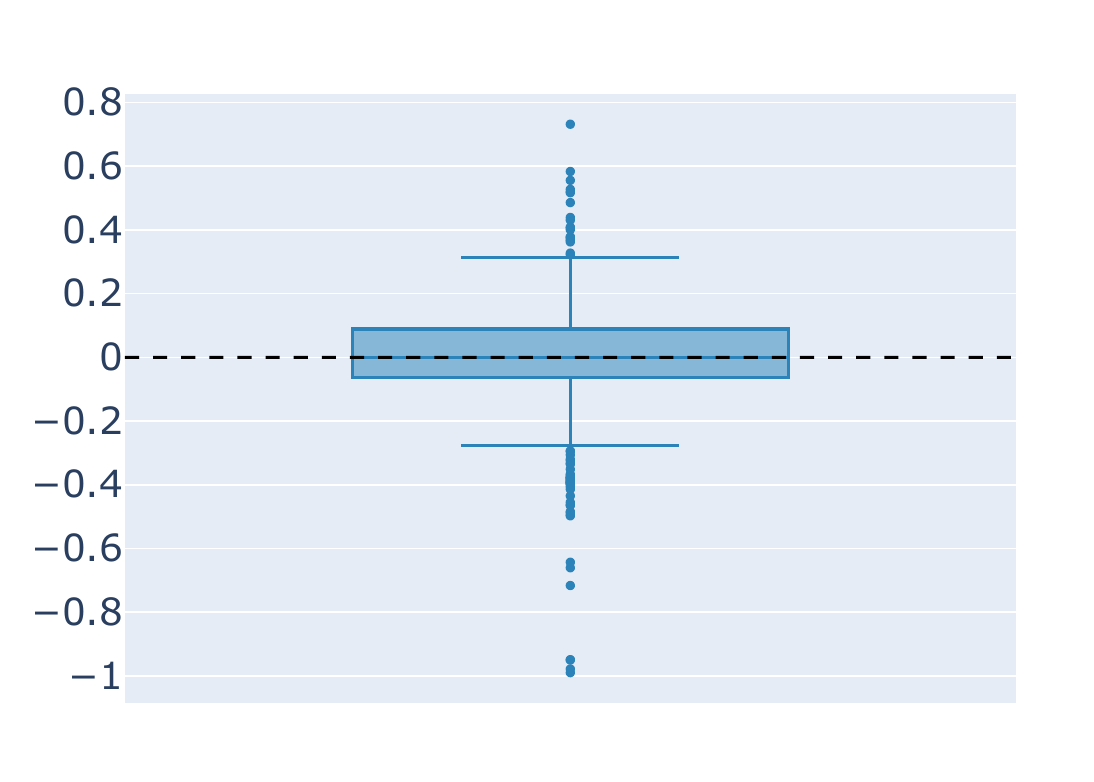} &
        \includegraphics[width=0.45\textwidth, valign=m]{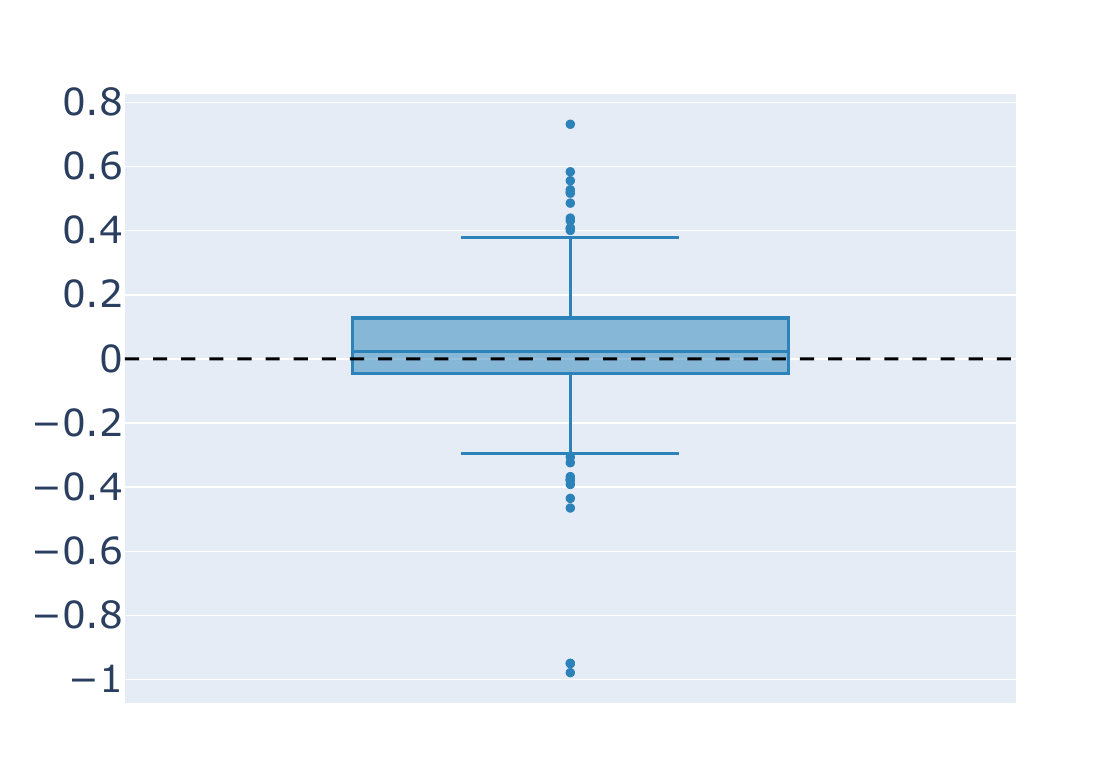} \\ [3ex]

        \rotatebox{90}{\revTwo{$\overline{\text{RE}}$ (All)}} &
        \includegraphics[width=0.45\textwidth, valign=m]{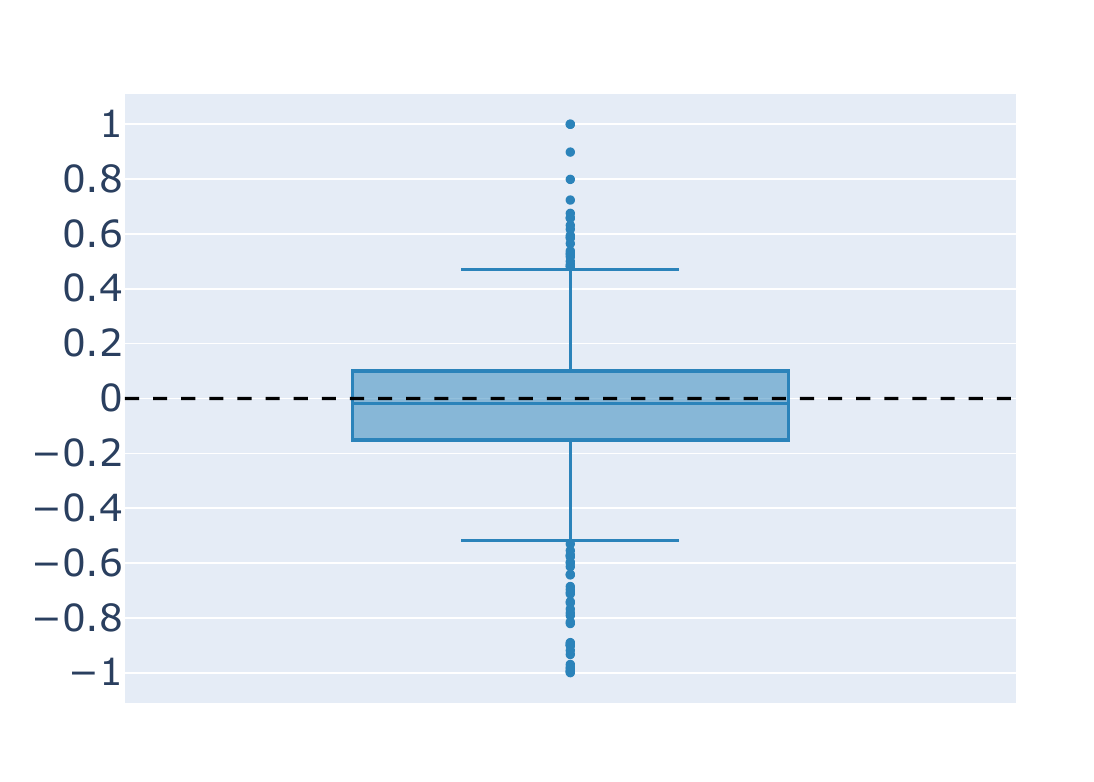} &
        \includegraphics[width=0.45\textwidth, valign=m]{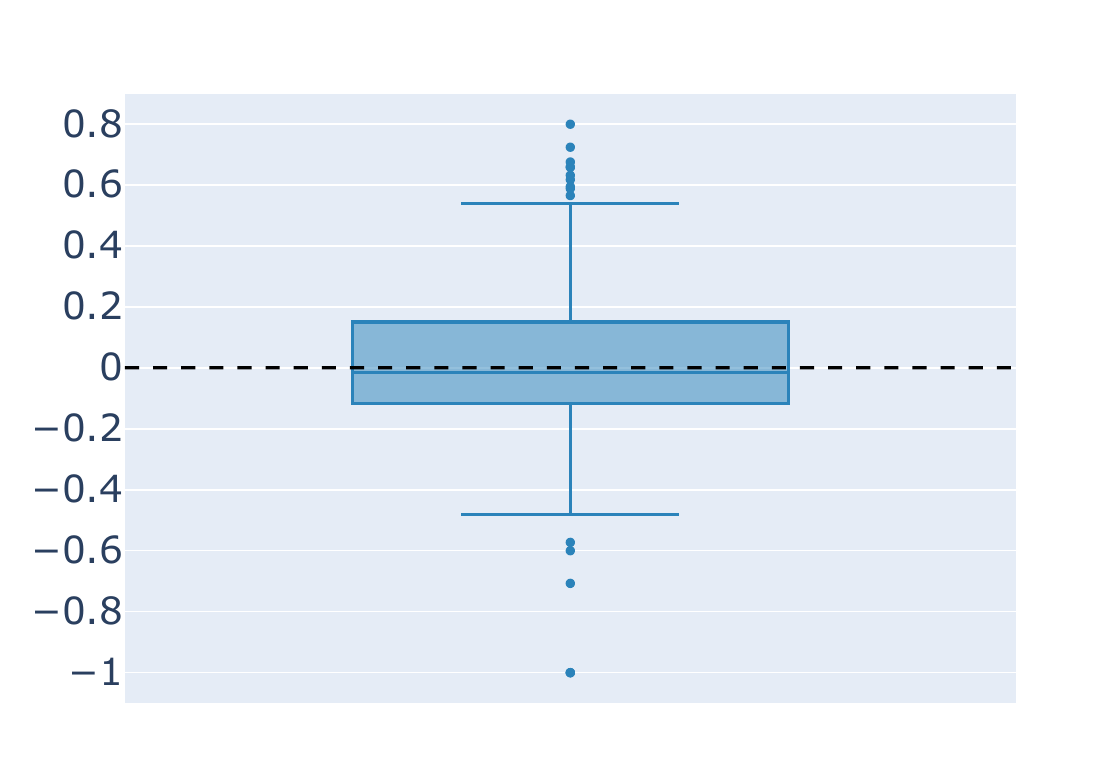} \\ [3ex]

        \rotatebox{90}{\revTwo{$\overline{\text{RE}}$ (Unlabeled)}} &
        \includegraphics[width=0.45\textwidth, valign=m]{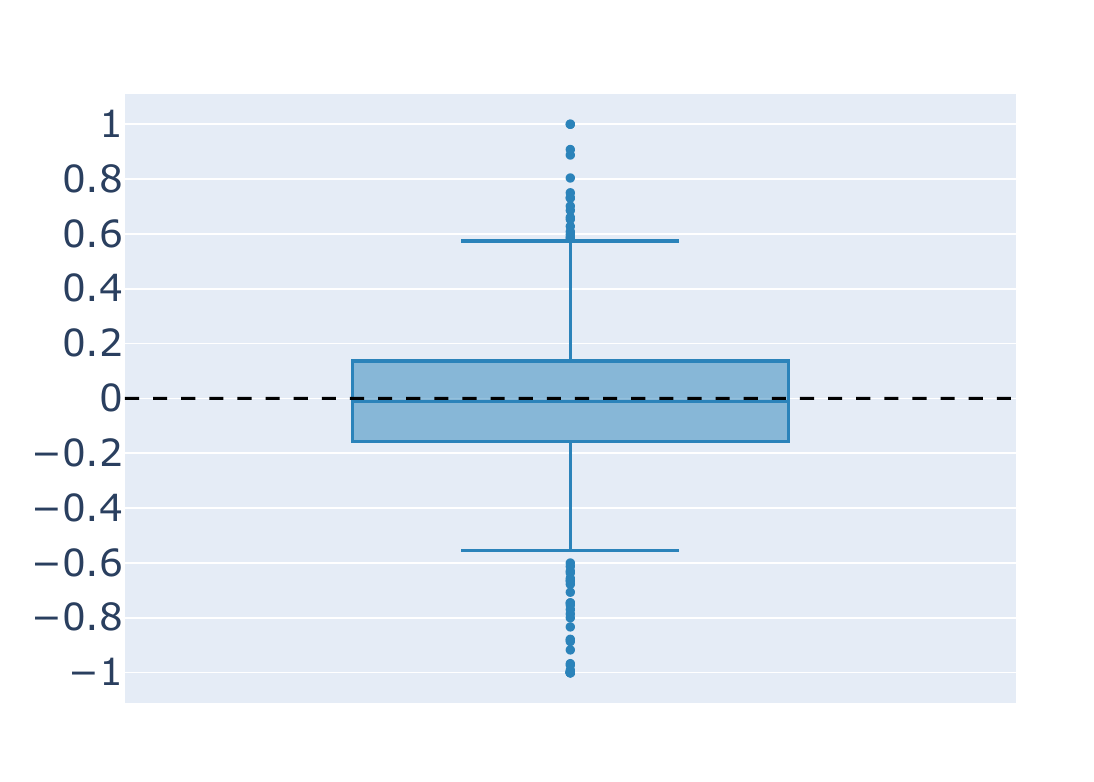} &
        \includegraphics[width=0.45\textwidth, valign=m]{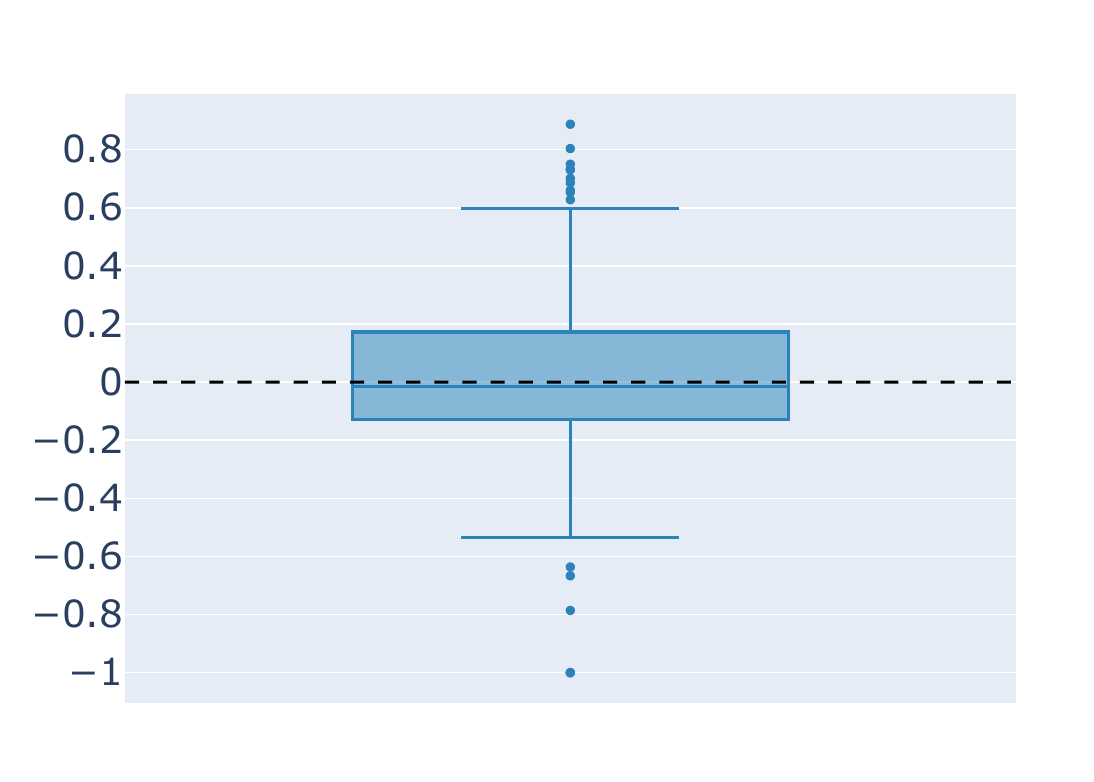}
    \end{tabular}

    \caption{\revTwo{Comparison of precision $\overline{\text{PR}}$ and recall $\overline{\text{RE}}$ for the simple sample approach as described in \eqref{comparOCTH2}. Rows distinguish between the entire data set and unlabeled data, while columns contrast all instances versus those where both approaches finished within the time limit (TL).}}
    \label{fig:PR&RESimpleSample}
\end{figure}


\end{document}